\documentclass{article}

\usepackage{mystyle}

\usepackage{authblk}
\title{A local epsilon version of Reed's Conjecture}

\author{Tom Kelly\thanks{Email: \texttt{t9kelly@uwaterloo.ca}} }

\author{Luke Postle\thanks{Partially supported by NSERC under Discovery Grant No.\ 2019-04304, the Ontario Early Researcher Awards program and the Canada Research Chairs program. Email: \texttt{lpostle@uwaterloo.ca}}}
\affil{Department of Combinatorics and Optimization\\
  University of Waterloo}

\date{June 24, 2021}

\begin{document}

\maketitle

\begin{abstract}
  In 1998, Reed conjectured that every graph $G$ satisfies $\chi(G) \leq \lceil \frac{1}{2}(\Delta(G) + 1 + \omega(G))\rceil$, where $\chi(G)$ is the chromatic number of $G$, $\Delta(G)$ is the maximum degree of $G$, and $\omega(G)$ is the clique number of $G$.  As evidence for his conjecture, he proved an ``epsilon version'' of it, i.e.\ that there exists some $\varepsilon > 0$ such that $\chi(G) \leq (1 - \varepsilon)(\Delta(G) + 1) + \varepsilon\omega(G)$.  It is natural to ask if Reed's conjecture or an epsilon version of it is true for the list-chromatic number.  In this paper we consider a ``local version'' of the list-coloring version of Reed's conjecture.  Namely, we conjecture that if $G$ is a graph with list-assignment $L$ such that for each vertex $v$ of $G$, $|L(v)| \geq \lceil \frac{1}{2}(d(v) + 1 + \omega(v))\rceil$, where $d(v)$ is the degree of $v$ and $\omega(v)$ is the size of the largest clique containing $v$, then $G$ is $L$-colorable.  Our main result is that an ``epsilon version'' of this conjecture is true, under some mild assumptions.

  Using this result, we also prove a significantly improved lower bound on the density of $k$-critical graphs with clique number less than $k/2$, as follows.  For every $\alpha > 0$, if $\varepsilon \leq \frac{\alpha^2}{1350}$, then if $G$ is an $L$-critical graph for some $k$-list-assignment $L$ such that $\omega(G) < (\frac{1}{2} - \alpha)k$ and $k$ is sufficiently large, then $G$ has average degree at least $(1 + \varepsilon)k$.  This implies that for every $\alpha > 0$, there exists $\varepsilon > 0$ such that if $G$ is a graph with $\omega(G)\leq (\frac{1}{2} - \alpha)\mathrm{mad}(G)$, where $\mathrm{mad}(G)$ is the maximum average degree of $G$, then $\chi_\ell(G) \leq \left\lceil (1 - \varepsilon)(\mad(G) + 1) + \varepsilon \omega(G)\right\rceil$.  It also yields an improvement on the best known upper bound for the chromatic number of $K_t$-minor free graphs for large $t$, by a factor of .99982.
\end{abstract}

\section{Introduction}

Let $G$ be a graph, and let $L = (L(v) : v\in V(G))$ be a collection of lists which we call {\em available colors}.  If each set $L(v)$ is non-empty, then we say that $L$ is a {\em list-assignment} for $G$.  If $k$ is an integer and $|L(v)|\geq k$ for every $v\in V(G)$, then we say that $L$ is a {\em $k$-list-assignment} for $G$.  An {\em $L$-coloring} of $G$ is a mapping $\phi$ with domain $V(G)$ such that $\phi(v)\in L(v)$ for every $v\in V(G)$ and $\phi(u)\neq\phi(v)$ for every pair of adjacent vertices $u,v\in V(G)$.  If $G$ has an $L$-coloring, then we say $G$ is \textit{$L$-colorable}.  We say that $G$ is {\em $k$-list-colorable}, or {\em $k$-choosable}, if $G$ has an $L$-coloring for every $k$-list-assignment $L$.  If $L(v) = \{1,\dots, k\}$ for every $v\in V(G)$, then we call an $L$-coloring of $G$ a {\em $k$-coloring}, and we say $G$ is {\em $k$-colorable} if $G$ has a $k$-coloring.  The \textit{chromatic number} of $G$, denoted $\chi(G)$, is the smallest $k$ such that $G$ is $k$-colorable.  The \textit{list-chromatic number} of $G$, denoted $\chi_\ell(G)$, is the smallest $k$ such that $G$ is $k$-list-colorable.

It is easy to see that for every graph $G$,
\begin{equation}
  \label{eq:intro-trivial-bounds}
  \omega(G) \leq \chi(G) \leq \chi_\ell(G) \leq \Delta(G) + 1,
\end{equation}
where $\omega(G)$ denotes the size of a largest clique in $G$ and $\Delta(G)$ denotes the maximum degree of a vertex in $G$.  If $G$ is a clique or an odd cycle then the upper bound in \eqref{eq:intro-trivial-bounds} is tight for both the chromatic number and list-chromatic number.  A classical theorem of Brooks \cite{B41} says that for the chromatic number, this is essentially the only case in which it is tight.
\begin{thm}[Brooks' Theorem \cite{B41}]
  If $G$ is a connected graph that is not a clique or odd cycle, then $\chi(G) \leq \Delta(G)$.
\end{thm}

In 1998, Reed \cite{R98} famously conjectured that, up to rounding, the chromatic number of a graph is at most the average of its clique number and maximum degree plus one.
\begin{conj}[Reed's Conjecture \cite{R98}]
  For every graph $G$,
  \begin{equation*}
    \chi(G) \leq \left\lceil\frac{1}{2}(\Delta(G) + 1 + \omega(G))\right\rceil.
  \end{equation*}
\end{conj}

As evidence for his conjecture, Reed~\cite{R98} proved that the chromatic nuumber of a graph is at most a weighted average of its clique number and maximum degree plus one.  We call this an ``epsilon version'' of Reed's Conjecture.
\begin{thm}[Reed \cite{R98}]\label{epsilon reeds}
  There exists $\varepsilon > 0$ such that for every graph $G$,
  \begin{equation*}
    \chi(G) \leq (1 - \varepsilon)(\Delta(G) + 1) + \varepsilon\omega(G).
  \end{equation*}
\end{thm}

Reed~\cite{R98} originally proved that Theorem~\ref{epsilon reeds} holds for graphs of sufficiently large maximum degree for $\varepsilon = 1.4\cdot 10^{-8}$.  In 2016, Bonamy, Perrett, and Postle~\cite{BPP16} improved this to $\varepsilon = \frac{1}{26}$.  Recently, Delcourt and Postle~\cite{DP17} (see~\cite{DP-eurocomb} for an extended abstract) improved this further to $\varepsilon = \frac{1}{13}$.  The blowup of a 5-cycle demonstrates that Theorem~\ref{epsilon reeds} does not hold for $\varepsilon \geq \frac{1}{2}$ and that the rounding in Reed's Conjecture is necessary.

It is natural to wonder if Brooks' Theorem or even Reed's Conjecture is true for the list-chromatic number.  We conjecture that for Reed's Conjecture this is the case.
\begin{conj}\label{list reed's conj}
  For every graph $G$,
  \begin{equation*}
    \chi_\ell(G) \leq \left\lceil\frac{1}{2}(\Delta(G) + 1 + \omega(G))\right\rceil.
  \end{equation*}
\end{conj}

The result of Delcourt and Postle~\cite{DP17} is actually proved for the list-chromatic number, implying an ``epsilon version'' of Conjecture~\ref{list reed's conj}.

In 1979, in one of the papers that introduced list-coloring, Erd\H os, Rubin, and Taylor \cite{ERT79} proved the following classical theorem.
\begin{thm}[Erd\H os, Rubin, and Taylor \cite{ERT79}]\label{local brooks}
  Let $G$ be a connected graph with list-assignment $L$.  If for every $v\in V(G)$, $|L(v)| \geq d(v)$, then $G$ is $L$-colorable, unless every block of $G$ is a clique or an odd cycle and for every $v\in V(G)$, $|L(v)| = d(v)$.
\end{thm}

Note that Theorem \ref{local brooks} implies that Brooks' Theorem is true for the list-chromatic number.  We consider Theorem \ref{local brooks} to be the archetype of what we call a ``local version.''  The main focus of this paper is the following conjecture, which we consider to be the natural ``local version'' of Reed's Conjecture and Conjecture \ref{list reed's conj}.

\begin{conj}[Local Version of Reed's Conjecture]\label{local reeds}
  If $G$ is a graph with list-assignment $L$ such that for every $v\in V(G)$,
  \begin{equation*}
    |L(v)| \geq \left\lceil \frac{1}{2}(d(v) + 1 + \omega(v))\right\rceil,
  \end{equation*}
  where $\omega(v)$ is the size of the largest clique containing $v$,
  then $G$ is $L$-colorable.
\end{conj}

Note that if true, Conjecture~\ref{local reeds} implies Reed's Conjecture and Conjecture~\ref{list reed's conj}.  As evidence for Conjecture~\ref{local reeds}, we prove an ``epsilon version'' of it, under certain mild assumptions.  The following is the main result of this paper.
\begin{thm}\label{main thm}
  Let $\varepsilon = \frac{1}{330}$.  If $G$ is a graph of sufficiently large maximum degree and $L$ is a list-assignment for $G$ such that for all $v\in V(G)$, $|L(v)| \geq \omega(v) + \log^{\logexp}(\Delta(G))$ 
  and 
  \begin{equation*}
    |L(v)| \geq (1 - \varepsilon)(d(v) + 1) + \varepsilon \omega(v),
  \end{equation*} 
  then $G$ is $L$-colorable.
\end{thm}

We prove Theorem~\ref{main thm} by proving structural properties of a ``minimum counterexample'' that enable us to then find an $L$-coloring using the probabilistic method.  The assumption in Theorem~\ref{main thm} that for each vertex $v$, $|L(v)| \geq \omega(v) + \log^{10}(\Delta(G))$, implies both that no vertex has a neighborhood that is ``too close'' to being a clique and that the minimum number of available colors for a vertex is sufficiently large.  As we will see, this in turn implies that a minimum counterexample to Theorem~\ref{main thm} has sufficiently large minimum degree.  It would be interesting to prove Theorem~\ref{main thm} with the hypothesis that $|L(v)| \geq \omega(v) + \log^{10}(\Delta(G))$ for each vertex $v$ replaced with the weaker assumption that the minimum degree of $G$ is at least $\log^{10}(\Delta(G))$. As we discuss in Section~\ref{overview section}, this may be possible to prove with an extension of our methods, at the expense of a worse value of $\varepsilon$.  However, since we use the probabilistic method, we do not believe our techniques could be extended to to prove Conjecture~\ref{local reeds} in full.

In Section~\ref{overview section}, we provide an overview of the proof of Theorem~\ref{main thm}, and Sections~\ref{metatheorem section}-\ref{concentration section} are devoted to its proof.  In order to prove Theorem~\ref{main thm}, we needed to develop a new version of Talagrand's ``Concentration Inequality,'' Theorem~\ref{exceptional talagrand's}, which we prove in Appendix~\ref{tala proof section}.  Our proof of Theorem~\ref{exceptional talagrand's} corrects a flaw in a version of Talagrand's Inequality in the book of Molloy and Reed~\cite[Talagrand's Inequality II]{MR00} (see Remark~\ref{talagrand mistake remark} in Section~\ref{concentration section}).

We now discuss some applications of our result.
\subsection{King's Conjecture}
In 2009, King \cite{K09} conjectured the following strengthening of Reed's Conjecture.
\begin{conj}[King \cite{K09}]\label{king's conj}
  For every graph $G$,
  \begin{equation*}
    \chi(G) \leq \max_{v\in V(G)}\left\lceil\frac{1}{2}(d(v) + 1 + \omega(v))\right\rceil.
  \end{equation*}
\end{conj}

King's idea behind Conjecture \ref{king's conj} was that a strengthened form of Reed's Conjecture may be easier to prove using induction.  For certain classes of graphs, this idea has been useful.  Using this and the structure theory of claw-free graphs of Chudnovsky and Seymour, King \cite{K09} proved that Reed's Conjecture is true for claw-free graphs.  The proof also appears in \cite{KR15}.  In 2013, Chudnovsky et al.\ \cite{CKPS13} proved that King's Conjecture holds for quasi-line graphs, and in 2015 King and Reed \cite{KR15} proved it for claw-free graphs with a 3-colorable complement.

Note that Conjecture \ref{local reeds}, if true, implies Conjecture \ref{king's conj}, even for list-coloring.  The first application of our main result is that it implies that an ``epsilon version'' of Conjecture \ref{king's conj} is true, assuming $G$ does not contain a clique of size within a factor of $\frac{329}{330} - o(1)$ of the maximum degree of $G$.  The following corollary follows easily from Theorem~\ref{main thm}.

\begin{cor}
  Let $\varepsilon \leq \frac{1}{330}$.  If $G$ is a graph of sufficiently large maximum degree such that $\omega(G) \leq (1 - \varepsilon)\Delta(G) - \log^{10}(\Delta(G))$, then
  \begin{equation*}
    \chi_\ell(G) \leq \max_{v\in V(G)}(1 - \varepsilon)(d(v) + 1) + \varepsilon\omega(v).
  \end{equation*}
\end{cor}
 
\subsection{Critical Graphs}
Now we discuss an application of Theorem~\ref{main thm}  to critical graphs.  A graph $G$ is \textit{$k$-critical} if $G$ is not $(k-1)$-colorable but every proper induced subgraph of $G$ is, and if $L$ is a list-assignment for $G$, then $G$ is \textit{$L$-critical} if $G$ is not $L$-colorable but every proper induced subgraph of $G$ is.  A list-assignment $L$ is \textit{$k$-uniform} if for every vertex $v$, $|L(v)| = k$.  We denote the \textit{average degree} of a graph $G$ by $\ad(G)$.  The average degree of critical graphs has been extensively studied.  Note that a $k$-critical graph has no vertex of degree less than $k - 1$, so the average degree of a $k$-critical graph is trivially at least $k-1$.  Much work has been devoted to improving this bound.  In a breakthrough result from 2014, Kostochka and Yancey \cite{KY14} proved the following lower bound on the number of edges in $k$-critical graphs.  

\begin{thm}[Kostochka and Yancey \cite{KY14}]\label{kostochka yancey edge bound}
  If $k \geq 4$ and $G$ is $k$-critical, then
  \begin{equation*}
    |E(G)| \geq \left\lceil\frac{(k + 1)(k - 2)|V(G)| - k(k - 3)}{2(k - 1)}\right\rceil.
  \end{equation*}
\end{thm}

Theorem~\ref{kostochka yancey edge bound} implies the following asymptotic lower bound on the average degree of $k$-critical graphs.
\begin{cor}[Kostochka and Yancey \cite{KY14}]\label{kostochka yancey ad bound}
  Let $k\geq 4$, and let $G$ be a $k$-critical graph on $n$ vertices.  Then as $n$ approaches infinity,
  \begin{equation*}
    \ad(G) \geq k - \frac{2}{k-1} - o(1).
  \end{equation*}
\end{cor}

Theorem~\ref{kostochka yancey edge bound} is tight for every $k$ for an infinite family of graphs, as shown by Ore \cite{O67}.  Therefore the asymptotic bound in Corollary~\ref{kostochka yancey ad bound} can not be improved.  Kostochka and Yancey asked if their bound can be improved by excluding certain subgraphs, such as cliques, and if similar results hold for list-coloring.  This was considered earlier by Kostochka and Stiebitz \cite{KS00}.  

\begin{thm}[Kostochka and Stiebitz \cite{KS00}]
  For every fixed $r$, if $G$ is $L$-critical for some $(k-1)$-uniform list-assignment $L$ and $\omega(G) \leq r$, then
  \begin{equation*}
    \ad(G) \geq 2k - o(k).
  \end{equation*}
\end{thm}

It is natural to not only consider graphs with bounded clique number but also graphs with clique number bounded by a function of $k$.  Theorem~\ref{main thm} implies that the bound in Corollary~\ref{kostochka yancey ad bound} can be improved for large $k$ if $G$ is an $L$-critical graph for some $k$-list-assignment $L$ and $G$ has no clique of size at least $k/2$, as follows.

\begin{thm}\label{critical thm}
  For every $\alpha > 0$, if $\varepsilon \leq \frac{\alpha^2}{1350}$ then the following holds.  If $G$ is an $L$-critical graph for some $k$-list-assignment $L$ such that $\omega(G) < (\frac{1}{2} - \alpha)k$ and $k$ is sufficiently large, then
  \begin{equation*}
    \ad(G) > (1 + \varepsilon)k.
  \end{equation*}
\end{thm}

\subsection{Maximum Average Degree}
The bound on the chromatic number supplied by Reed's Conjecture can be viewed as the average of the lower and upper bounds provided in \eqref{eq:intro-trivial-bounds}, as previously mentioned.  However, the upper bound in \eqref{eq:intro-trivial-bounds} can easily be improved by replacing $\Delta(G)$ with $\lfloor \mad(G)\rfloor$, where $\mad(G) = \max_{H\subseteq G}\ad(H)$, the \textit{maximum average degree} of $G$.  In the spirit of Reed's Conjecture, we conjecture the following which, if true, implies Reed's Conjecture.

\begin{conj}\label{mad conjecture}
  For every graph $G$,
  \begin{equation*}
    \chi_\ell(G) \leq \left\lceil\frac{1}{2}\left(\mad(G) + 1 + \omega(G)\right)\right\rceil.
  \end{equation*}
\end{conj}

Note that Conjecture~\ref{mad conjecture}, if true, would be tight for $K_{2,4}$, since $\chi_\ell(K_{2,4}) = 3$.  More generally, $\chi_\ell(K_{t, t^t}) = t + 1$ and $\mad(K_{t, t^t}) = 2t^t/(1 + t^{t-1}) \leq 2t$, so the graphs $K_{t, t^t}$ provide an infinite family for which the difference of the right and left side of the inequality in Conjecture~\ref{mad conjecture} is at most one.

Another application of Theorem~\ref{main thm} is an ``epsilon version'' of Conjecture~\ref{mad conjecture} for graphs with clique number less than half their maximum average degree.
\begin{thm}\label{epsilon mad}
  For every $\alpha > 0$, there exists $\varepsilon > 0$ such that the following holds.  For every graph $G$ such that $\omega(G) \leq (\frac{1}{2} - \alpha)\mad(G)$,
  \begin{equation*}
    \chi_\ell(G) \leq \left\lceil(1 - \varepsilon)(\mad(G) + 1) + \varepsilon\omega(G)\right\rceil.
  \end{equation*}
\end{thm}
Theorem~\ref{epsilon mad} follows easily from Theorem~\ref{critical thm}.  We include the proof in Section~\ref{critical section}.

\subsection{$K_t$-minor free graphs}
We conclude this section with an application of Theorem~\ref{critical thm} to Hadwiger's conjecture, which is considered one of the most important open problems in graph theory.  Hadwiger~\cite{H43} conjectured in 1943 that if a graph has no $K_{t+1}$-minor, then it has chromatic number at most $t$.  The best known upper bound on the chromatic number of $K_t$-minor free graphs to date uses the fact that the chromatic number of a graph is at most its maximum average degree, combined with the following theorem of Thomason~\cite{T01} providing a tight upper bound on the average degree of $K_t$-minor free graphs.\footnote{Following acceptance of this paper for publication in JCTB, further improvements were made by Norin and Song~\cite{NS19} and Postle~\cite{P19}.}
\begin{thm}[Thomason \cite{T01}]\label{clique minor average degree bound}
  If $G$ is a graph with no $K_t$-minor, then
  \begin{equation*}
    \ad(G) \leq (\gamma + o(1))t\sqrt{\log t},
  \end{equation*}
  where $\gamma = 0.63817...$ is an explicit constant.
\end{thm}

By combining Theorem~\ref{critical thm} with Theorem~\ref{clique minor average degree bound}, we can improve the best known upper bound on the chromatic number of $K_t$-minor free graphs by a constant factor, as follows.
\begin{cor}
  If $G$ is a graph with no $K_t$-minor, then
  \begin{equation*}
    \chi_\ell(G) \leq (.99982\cdot\gamma + o(1))t\sqrt{\log t},
  \end{equation*}
  where $\gamma$ is the explicit constant from Theorem~\ref{clique minor average degree bound}.
\end{cor}
\begin{proof}
  It suffices to show that for every $\xi > 0$, if $k_t = .99982(\gamma + \xi)t\sqrt{\log t},$ then for sufficiently large $t$, every $K_t$-minor free graph is $k_t$-list-colorable. Suppose not.  Then there exists a graph $G$ with no $K_t$-minor that is $L$-critical for some $k_t$-list-assignment $L$ where $k_t  \geq 1000t$.  Using Theorem~\ref{clique minor average degree bound}, we may assume $\ad(G) \leq k_t/.99982$.
  
  Let $\alpha = 499/1000$ and $\varepsilon = \alpha^2/1350$.  Since $\omega(G) < t$, $\omega(G) < (\frac{1}{2} - \alpha)k_t$.  Since $G$ is $L$-critical, by Theorem~\ref{critical thm}, $\ad(G) > (1 + \varepsilon)k_t$.  But $1 + \varepsilon \geq 1/.99982$, a contradiction.
\end{proof}

\section{Overview of the Proof of Theorem~\ref{main thm} and Outline of the Paper}\label{overview section}
The following definition will be useful.
\begin{define}
  Let $G$ be a graph.  For each $v\in V(G)$ we let
  \begin{equation*}
    \Gap_G(v) = d(v) + 1 - \omega(v),
  \end{equation*}
  and if $L$ is a  list-assignment for $G$, we let
  \begin{equation*}
    \Save_L(v) = d(v) + 1 - |L(v)|.
  \end{equation*}
\end{define}
If the graph $G$ or list-assignment $L$ is clear from the context, we may omit the subscript $G$ or $L$ in $\Gap$ and $\Save$, respectively.
Note that the conditions of Theorem~\ref{main thm} imply that for each vertex $v\in V(G)$, $\Gap(v) - \Save(v)\geq \log^{\logexp}(\Delta(G))$ and $\Save(v) \leq \varepsilon\Gap(v)$.

First we discuss our strategy for proving Theorem~\ref{main thm}.  We use a variant of a technique called the ``naive coloring procedure,'' given its name by Molloy and Reed~\cite{MR00}.  Essentially, we analyze a random partial coloring of a graph and prove that with nonzero probability this partial coloring can be extended deterministically to a coloring of the whole graph.  The random partial coloring is described formally in Definition~\ref{random coloring procedure}.  After the random partial coloring, we let $G'$ be the subgraph induced by $G$ on the vertices that are not colored, and we let $L'$ be a list-assignment for $G$ so that any $L'$-coloring of $G'$ can be combined with the random partial coloring to obtain an $L$-coloring of $G$.  We prove that with nonzero probability $G'$ is $L'$-colorable.  To do this, we would like to show that with high probability, for every vertex $v\in V(G')$, $|L'(v)| > d_{G'}(v)$, i.e.\ that $\Save_{L'}(v) \leq 0$.  However, this is not the case.  In fact, it may be likely that $\Save_{L'}(v) = \Save_{L}(v)$.  For example, the neighborhood of a vertex $v$ may form $\sqrt{d(v)}$ cliques, while for the list-assignment $L$, the vertices in each clique have the same list of available colors and vertices in different cliques have disjoint lists of available colors.  Nevertheless, if a vertex $v$ has many neighbors with at least as many available colors, we are able to show that $\Save_{L'}(v) < \Save_L(v)$.  This motivates the following definitions.

\begin{define}
  Let $\alpha$ be some constant to be determined later.  Let $G$ be a graph with list-assignment $L$, let $v\in V(G)$, and let $u\in N(v)$.
  \begin{itemize}
  \item If $|L(u)| < |L(v)|$, then we say $u$ is a \textit{subservient neighbor} of $v$.
  \item If $|L(u)| \in [|L(v)|, (1 + \alpha)|L(v)|)$, then we say $u$ is an \textit{egalitarian neighbor} of $v$.
  \item If $|L(u)| \geq (1 + \alpha)|L(v)|$, then we say $u$ is a \textit{lordlier neighbor} of $v$.
  \end{itemize}
  For convenience, we will let $\aberrant(v)$ denote the set of lordlier neighbors of $v$, $\egal(v)$ denote the set of egalitarian neighbors of $v$, and $\subserv(v)$ denote the set of subservient neighbors of $v$.
\end{define}

\begin{define}
  Let $\beta$ be some constant to be determined later.  Let $G$ be a graph with list-assignment $L$, let $v\in V(G)$, and let $u$ be an egalitarian neighbor of $v$.
  \begin{itemize}
  \item If $|L(u)| < |L(v)| + \beta\Gap(v)$, then we say $u$ is a \textit{strongly egalitarian neighbor} of $v$.
  \item If $|L(u)| \geq |L(v)| + \beta\Gap(v)$, then we say $u$ is a \textit{weakly egalitarian neighbor} of $v$.
  \end{itemize}
  For convenience, we will let $\strongEgal(v)$ denote the set of strongly egalitarian neighbors of $v$, $\weakEgal(v)$ denote the set of weakly egalitarian neighbors of $v$, and $\notEgal(v) = N(v) - \egal(v)$.
\end{define}

If a vertex $v$ has many subservient neighbors, then we say $v$ is \textit{lordly}.  The names ``subservient'', ``egalitarian'', and ``lordlier'' neighbors are evocative of feudalism in medeival Europe, where power is analogous to list size.  As mentioned previously, if $v$ is a lordly vertex, we are unable to guarantee that $\Save_{L'}(v) < \Save_L(v)$ for certain list-assignments for $v$'s subservient neighbors.  We resolve this issue by coloring vertices before their subservient neighbors when finding an $L'$-coloring, thus giving ``priority'' to the lordly vertices.

A lordlier neighbor also has the power to choose from more colors.
If $v$ has many lordlier neighbors or weakly egalitarian neighbors, then it is likely that after the random partial coloring $v$ has many neighbors receiving a color not in $L(v)$.  If $v$ has many egalitarian neighbors, then it is likely that after the random partial coloring there are many colors assigned to multiple neighbors of $v$.  In both cases, $\Save_{L'}(v) < \Save_L(v)$.

A common technique in coloring is to attempt to greedily color a vertex of smallest degree, since fewer neighbors means fewer potential color conflicts.  However, for our ``local version,'' this technique is not so useful because vertices of lower degree also have fewer available colors.  Our trick to finding an $L'$-coloring of $G'$ is to order the vertices of $G'$ by the size of their list in $L$, from greatest to least, and color greedily, which may seem counterintuitive.  This works because we are able to guarantee for every vertex $v\in V(G')$, that $\Save_{L'}(v)$ is smaller than the number of neighbors of $v$ in $G'$ that will be colored after $v$ in this ordering, and thus $|L'(v)|$ is larger than the number of neighbors of $v$ in $G'$ that will be colored before $v$ in this ordering.  

For each vertex $v$, after an application of our naive coloring procedure, we refer to the number of neighbors of $v$ receiving a color not in $L(v)$, plus the multiplicity less 1 of each color in $L(v)$ assigned to multiple neighbors, plus the number of uncolored subservient neighbors of $v$ as the ``savings'' for $v$.  In order to prove Theorem~\ref{main thm}, we first prove Theorem~\ref{metatheorem}, which essentially says that it suffices to show that the expected savings for each vertex is at least $\Save_L(v)$ and is also sufficiently large.  Here ``sufficiently large'' means $\mathrm{poly}\log\Delta$, which we need in order to show that the savings for each vertex is sufficiently close to its expectation with probability inverse to a polynomial in $\Delta$, in which case we can apply the Lov\'{a}sz Local Lemma to guarantee an outcome for which the savings for every vertex is close to its expectation.

Using Theorem~\ref{metatheorem} it suffices to show that the expected savings for a vertex $v$ is $\Omega(\Gap(v))$.  If the savings for each vertex $v$ is at least $\Omega(\Gap(v))$, then the condition in Theorem~\ref{main thm} that $\Save(v) \leq \varepsilon\Gap(v)$ guarantees that the savings for $v$ is at least $\Save(v)$ if $\varepsilon$ is small enough.  Moreover, the technical condition in Theorem~\ref{main thm} that $\Gap(v) - \Save(v) \geq \log^{10}(\Delta(G))$ ensures that the savings are large enough to obtain concentration.  In order to prove Theorem~\ref{main thm} without this latter condition with our methods, it is necessary to find a way to show that a vertex $v$ with $\Gap(v) = O(\log^{10}(\Delta(G)))$ still has savings at least on the order of $\log^{10}(\Delta(G))$.

\subsection{Outline of the Paper}

We prove Theorem~\ref{critical thm} and Theorem~\ref{epsilon mad} in Section~\ref{critical section}.  The rest of the paper is devoted to the proof of Theorem~\ref{main thm}.

In Section~\ref{metatheorem section}, we formalize the previous discussion on the ``naive coloring procedure'' and prove Theorem~\ref{metatheorem}, which could be considered a ``metatheorem.''  We also use Theorem~\ref{metatheorem} in a follow-up paper~\cite{KP19}.  In order to prove Theorem~\ref{metatheorem}, we need to show that the savings for each vertex is concentrated around its expectation.  Lemma~\ref{concentration lemma} makes this precise.  We prove Lemma~\ref{concentration lemma} in Section~\ref{concentration section}.  

Before proving Theorem~\ref{main thm}, in Section~\ref{structure section} we prove that a minimum counterexample to Theorem~\ref{main thm} has some desirable structure.  The main result of Section~\ref{structure section} is Theorem~\ref{structure thm}, which says that in a minimum counterexample $G$, each $v\in V(G)$ either has many non-adjacent egalitarian neighbors, many lordlier neighbors, or many subservient neighbors.  The idea to separate the strongly egalitarian neighbors from the weakly egalitiarian neighbors is crucial here, because the weakly egalitarian neighbors of a vertex $v$ are also likely to receive a color not in $L(v)$.

In Section~\ref{main proof section}, we exploit this structure to lower bound the expected value of each type of savings in Lemmas~\ref{lordly save}, \ref{aberrant save}, and \ref{sparse save}.  Using these lemmas in conjunction with Theorem~\ref{metatheorem}, we prove Theorem~\ref{main thm} in Section~\ref{main proof section}.

In Section~\ref{concentration section} we prove Lemma~\ref{concentration lemma}, which completes the proof of Theorem~\ref{metatheorem}.  In order to prove this lemma we needed to develop a new ``concentration inequality,'' Theorem~\ref{exceptional talagrand's}, which provides sufficient conditions for a random variable to be concentrated around its expectation with high probability.  Theorem~\ref{exceptional talagrand's} is similar to results provided in \cite{BJ15, MR00}, but those did not work for our purposes.  We prove Theorem~\ref{exceptional talagrand's} in Appendix~\ref{tala proof section} using Talagrand's inequality.

As mentioned above, in Section~\ref{critical section} we prove Theorem~\ref{critical thm} and Theorem~\ref{epsilon mad}.
 
\section{The Local Naive Coloring Procedure}\label{metatheorem section}

The main result of this section is Theorem~\ref{metatheorem}, which gives sufficient conditions for our naive coloring procedure to be extended to a coloring of the whole graph.  Namely, we need that the expected ``savings'' for each vertex is at least $\Save_L(v)$ and is sufficiently large.  Before we can state Theorem~\ref{metatheorem}, we need to formalize our naive coloring procedure.

In this section, we let $G$ be a graph with list-assignment $L$, $\varepsilon,\sigma, \in [0, 1)$, $\rho \in [0, 1]$, and $\ordering$ be a partial ordering of $V(G)$.  When we apply Theorem~\ref{metatheorem} to prove Theorem~\ref{main thm} in Section~\ref{main proof section}, we let $\sigma$ be 0 and for $u,v\in V(G)$, we have $u\ordering v$ if $|L(u)| < |L(v)|$.  We include these parameters because we plan to use Theorem~\ref{metatheorem} in a follow-up paper in which $\sigma > 0$ and $\ordering$ is different.  In order to demonstrate how $\sigma$ will be used, we need the following definition.

\begin{define}
  For each $v\in V(G)$ and $u\in N(v)$, we say $u$ is a \textit{$\sigma$-egalitarian neighbor} of $v$ if $u$ has at least $(1 - \sigma)|L(v)|$ available colors.  We let $\egal_\sigma(v)$ denote the set of $\sigma$-egalitarian neighbors of $v$.
\end{define}

As we will see in Section~\ref{concentration section} and as alluded to in Section~\ref{overview section}, we cannot prove that the number of colors assigned to multiple neighbors of $v$ that are not $\sigma$-egalitarian is concentrated around its expectation.

To simplify our probabilistic analysis, we use a generalization of list-coloring known as \textit{correspondence coloring}, first introduced by Dvo\v{r}\'{a}k and Postle~\cite{DP15}.  Using correspondence coloring also helps improve the value of $\varepsilon$ in Theorem \ref{main thm}, because we can assume egalitarian neighbors of a vertex $v$ have at least $|L(v)|$ colors in common, thus making it more likely that a color is assigned to more than one of them.  Recall that $L$ is a list-assignment for $G$.

\begin{define}\label{correspondence coloring}\leavevmode
  \begin{itemize}
  \item If $M$ is a function defined on $E(G)$ where for each $e=uv\in E(G)$, $M_e$ is a matching of $\{u\}\times L(u)$ and $\{v\}\times L(v)$, then $(L, M)$ is a \textit{correspondence assignment} for $G$.  If for each $e=uv\in E(G)$ the matching $M_e$ saturates at least one of $\{u\}\times L(u)$ or $\{v\}\times L(v)$, then we say $(L, M)$ is \textit{total}.
	
  \item An {\em $(L, M)$-coloring} of $G$ is a function $\phi:V(G)\rightarrow\mathbb N$ such that $\phi(u)\in L(u)$ for every $u\in V(G)$, and for every $e=uv\in E(G)$, $(u, \phi(u))(v, \phi(v))\notin M_e$.  If $G$ has an $(L, M)$-coloring, then $G$ is {\em $(L, M)$-colorable}.
  \end{itemize}
\end{define}
One defines a \textit{$k$-correspondence assignment} and the \textit{correspondence chromatic number} in the natural way, but we do not need these terms.  For convenience, if $uv\in E(G)$, $c_1\in L(u)$, $c_2\in L(v)$, and $(u, c_1)(v, c_2)\in M_{uv}$, we will just say $c_1c_2\in M_{uv}$.  Note that if for each $e=uv\in E(G)$ and $c\in L(u)\cap L(v)$, $cc\in M_{uv}$, then an $(L, M)$-coloring is an $L$-coloring.
 
For the remainder of this section, let $(L, M)$ be a correspondence assignment for $G$.  We will actually define our naive coloring procedure for correspondence coloring.  First, we need some definitions.
\begin{define}\leavevmode
  \begin{itemize}
  \item We say a \textit{naive partial $(L, M)$-coloring} of $G$ is a pair $(\phi, \uncolvtcs)$ where $\phi : V(G)\rightarrow\mathbb N$ such that $\phi(u)\in L(u)$ for every $u\in V(G)$ and $\uncolvtcs\subseteq V(G)$ is a set of \textit{uncolored vertices} such that $\phi|_{V(G)-\uncolvtcs}$ is an $(L, M)$-coloring of $G - \uncolvtcs$.
  \item If $(\phi, \uncolvtcs)$ is a naive partial $(L, M)$-coloring of $G$, for each $v\in \uncolvtcs$, let
    \begin{equation*}
      \newList(v) = L(v) \backslash \{c\in L(v) : \exists u\in N(v)\backslash V(G'), c\phi(u)\in M_{vu}\}
    \end{equation*}
    and for each $uv\in E(G[U])$, let $\newmatching_{uv}$ be the matching induced by $M_{uv}$ on $\{u\}\times \newList(u)$ and $\{v\}\times \newList(v)$.
  \end{itemize}
\end{define}

If $(\phi, \uncolvtcs)$ is a naive partial $(L, M)$-coloring of $G$, then we call a vertex $v$ \textit{uncolored} if it is in $\uncolvtcs$, and otherwise we call it \textit{colored}.

The following proposition is self-evident.
\begin{prop}\label{extending partial coloring}
  If $(\phi, \uncolvtcs)$ is a naive partial $(L, M)$-coloring of $G$ and $G[\uncolvtcs]$ is $(\newList, \newmatching)$-colorable, then $G$ is $(L, M)$-colorable.
\end{prop}

The following is a variant of the naive coloring procedure, but it is not the one we use in Theorem~\ref{metatheorem}.  Recall that $\rho \in [0, 1]$.
\begin{define}\label{local-naive-coloring-def}
  The \textit{local naive random coloring procedure with activation probability $\rho$} samples a random naive partial $(L, M)$-coloring $(\phi, U)$ and a set of \textit{activated vertices} $A$ in the following way.
  For each $v\in V(G)$,
  \begin{enumerate}
  \item let $v\in A$ independently at random with probability $\rho$,
  \item choose $\phi(v)\in L(v)$ independently and uniformly at random, and
  \item\label{uncoloring-rule} let $\uncolvtcs = (V(G)\setminus A) \cup \uncolvtcs'$, where $v\in\uncolvtcs'$ if there exists $u\in N(v)\cap A$ such that $|L(u)| \geq |L(v)|$ and $\phi(u)\phi(v) \in M_{uv}$.
  \end{enumerate}
\end{define}

We also consider the following proposition to be self-evident.
\begin{prop}\label{keep probability}
  If $(\phi, \uncolvtcs)$ is a random naive partial $(L, M)$-coloring sampled using the local naive random coloring procedure with activation probability $\rho$, then for each $v\in V(G)$ and $c \in L(v)$,
  \begin{equation*}
    \ProbCond{v\notin \uncolvtcs}{\phi(v) = c} \geq \rho\prod_{\{u\in N(v) : |L(u)| \geq |L(v)|\}}\left(1 - \frac{\rho}{|L(u)|}\right).
  \end{equation*}
\end{prop}

Recall that $\varepsilon\in[0, 1)$.  Let $K_{\varepsilon, \rho} = .999\rho e^{\frac{-\rho}{1 - \varepsilon}}$.  We need the following proposition.
\begin{prop}\label{keep probability is large}
  There exists $\delta = \delta(\varepsilon)$ such that the following holds.  Let $(\phi, \uncolvtcs)$ be a random naive partial $(L, M)$-coloring sampled using the local naive random coloring procedure with activation probability $\rho$.  If for each $v\in V(G)$, $|L(v)| \geq (1 - \varepsilon)d(v)$ and $G$ has minimum degree at least $\delta$, then for each $v\in V(G)$ and $c \in L(v)$,
  \begin{equation*}
    \ProbCond{v\notin \uncolvtcs}{\phi(v) = c} \geq K_{\varepsilon, \rho}.
  \end{equation*}
\end{prop}
\begin{proof}
  By Proposition~\ref{keep probability},
  \begin{equation*}
    \ProbCond{v\notin\uncolvtcs}{\phi(v) = c} \geq \rho \left(1 - \frac{\rho}{(1 - \varepsilon)d(v)}\right)^{d(v)} \geq \rho\left(1 - \frac{\rho^2}{(1 - \varepsilon)^2d(v)}\right)e^{-\frac{\rho}{1 - \varepsilon}}.
  \end{equation*}
  We let $\delta(\varepsilon) = 1000 / (1 - \varepsilon)^2$, and the result follows.
\end{proof}

Now we introduce the random coloring procedure that we use in Theorem~\ref{metatheorem}, which is slightly easier to analyze than the local naive random coloring procedure.

\begin{define}\label{random coloring procedure}
  If for each $v\in V(G)$, $|L(v)| \geq (1 - \varepsilon)d(v)$ and $G$ has minimum degree at least $\delta(\varepsilon)$ (as in Proposition~\ref{keep probability is large}), then the \textit{local naive random coloring procedure with activation probability $\rho$ and $\varepsilon$-equalizing coin-flips} samples a random naive partial $(L, M)$-coloring $(\phi, \uncolvtcs)$ (and a set of activated vertices $A$) in the following way.
  \begin{enumerate}
  \item Sample a random naive partial $(L, M)$-coloring $(\phi, \uncolvtcs')$ and a set $A$ of activated vertices using the local naive random coloring procedure with activation probability $\rho$,
  \item for each $v\in V(G)$ and $c \in L(v)$, conduct a ``coin flip'' for $v$ and $c$ that is ``heads'' with probability $1 - K_{\varepsilon, \rho}/\ProbCond{v\notin \uncolvtcs'}{\phi(v) = c}$, and
  \item let $\uncolvtcs = \uncolvtcs' \cup \uncolvtcs''$, where $v \in U''$ if the coin flip for $v$ and $\phi(c)$ is heads.
  \end{enumerate}
\end{define}

For the remainder of this section, we assume $G$ and $(L, M)$ satisfy the assumptions of Definition~\ref{random coloring procedure}, and we let $(\phi, \uncolvtcs)$ be a random naive partial $(L, M)$-coloring and $A$ a set of activated vertices sampled using the the local naive random coloring procedure with activation probability $\rho$ and $\varepsilon$-equalizing coin-flips.  The following proposition shows why the $\varepsilon$-equalizing coin-flips are useful.

\begin{prop}\label{equalized keep probability}
  For each $v\in V(G)$ and $c \in L(v)$,
  \begin{equation*}
    \ProbCond{v\notin \uncolvtcs}{\phi(v) = c} = K_{\varepsilon, \rho}.
  \end{equation*}
\end{prop}
\begin{proof}
  Let $\uncolvtcs = \uncolvtcs'\cup \uncolvtcs''$ as in Definition~\ref{random coloring procedure}.  Note that $\ProbCond{v\notin \uncolvtcs}{\phi(v) = c} = \ProbCond{v\notin \uncolvtcs''}{\phi(v) = c} \cdot \ProbCond{v \notin \uncolvtcs'}{\phi(v) = c}$.  By the choice of $\uncolvtcs''$, we have $\ProbCond{v\notin \uncolvtcs''}{\phi(v) = c} = K_{\varepsilon, \rho} / \ProbCond{v\notin \uncolvtcs'}{\phi(v) = c}$, and the result follows.
\end{proof}

Recall that a vertex $u$ is a $\sigma$-egalitarian neighbor of a vertex $v$ if $|L(u)| \geq (1 - \sigma)|L(v)|$.  Recall also that $\ordering$ is a partial ordering of $V(G)$.  We can now formalize what we mean by the ``savings'' for each vertex, as follows.
\begin{define} For each $v\in V(G)$, we define the following random variables.  
  \begin{itemize}
  \item Let $\aberrance_{v,\sigma}$ count the number of colored $\sigma$-egalitarian neighbors $u$ of $v$ such that $\phi(u)$ is not matched by $M_{uv}$.
  \item Let $\pairs_{v, \sigma}$ and $\trips_{v, \sigma}$ count the number of pairs and triples respectively of colored $\sigma$-egalitarian neighbors of $v$ that receive colors that are matched to the same color in $L(v)$.
  \item Let $\unact_{v, \ordering}$ count the number of non-activated neighbors $u$ of $v$ such that $u\ordering v$.
  \item Let $\savings_{v, \sigma, \ordering} = \aberrance_{v, \sigma} + \unact_{v, \ordering} + \pairs_{v, \sigma} - \trips_{v, \sigma}$.
  \end{itemize}
        
  More precisely, we have that
  \begin{align*}
    &\aberrance_{v,\sigma} = |\{u\in \egal_\sigma(v) \setminus \uncolvtcs  : \phi(u)\notin V(M_{uv})\}|,\\
    \begin{split}
      &\pairs_{v, \sigma} = |\{x,y\in \egal_\sigma(v)\setminus \uncolvtcs, c\in L(v) : \phi(x)c\in M_{xv}\text{ and }\phi(y)c\in M_{yv}\}|,
    \end{split}\\
    \begin{split}
      &\trips_{v, \sigma} = |\{x,y,z\in \egal_\sigma(v)\setminus \uncolvtcs, c\in L(v) :  \phi(x)c\in M_{xv}, \phi(y)c\in M_{yv}, \text{ and }\phi(z)c\in M_{zv}\}|, \text{ and}
    \end{split}\\
      &\unact_{v, \ordering} = |\{u\in N(v)\setminus A : u \ordering v\}|.\\
  \end{align*}
\end{define}

\begin{remark}
  In the journal version of this paper, there are a few mistakes in the preceding part of this section that we have corrected.  In Sections~\ref{main proof section} and~\ref{concentration section} and in the remainder of this section, we make minor adjustments to account for these changes.  We descrbe these changes below.  
  \begin{enumerate}
  \item  In the journal version of this paper, Proposition~\ref{keep probability} is incorrect.  We correct this mistake by reversing the inequality in Step~\ref{uncoloring-rule} of Definition~\ref{local-naive-coloring-def}.  However, in the previous version, instead of $\unact_{v, \ordering}$, we used the random variable $\subservience_{v, \ordering}$ where $\subservience_{v, \ordering}(\phi, \uncolvtcs) = |\{u \in N(v) \cap U : u \ordering v\}|$, and with the change to Definition~\ref{local-naive-coloring-def}, $\subservience_{v, \ordering}$ is no longer concentrated around its expectation.  Thus, this version introduces activation probabilities, a commonly used technique, in order to define $\unact_{v, \ordering}$.  We replace $\subservience_{v, \ordering}$ with $\unact_{v, \ordering}$ throughout the paper with minimal changes.  Similarly, in the previous version, instead of $\aberrance_{v, \sigma}$, we used $\aberrance_v$ where $\aberrance_v(\phi, U) = |\{u\in N(v) \setminus \uncolvtcs  : \phi(u)\notin V(M_{uv})\}|$, and with the change to Definition~\ref{local-naive-coloring-def}, $\aberrance_v$ is no longer concentrated around its expectation.  Nevertheless, $\aberrance_{v, \sigma}$ is, and we can replace $\aberrance_v$ with $\aberrance_{v, \sigma}$ throughout the paper with minimal changes.

  \item In Definition~\ref{random coloring procedure} of the journal version of this paper, we conduct only one $\varepsilon$-equalizing coin flip for each vertex.  However, with this definition, if $u \in N(v)$, then the events ``$\phi(u) \notin V(M_{uv})$'' and ``$u\notin U$'' (as in the definition of $\aberrance_{v, \sigma}$) are not necessarily independent, and likewise, if $x,y\in\egal_\sigma(v)$ and $c_x\in L(x), c_y\in L(y)$ are colors such that $cc_x \in M_{vx}$ and $cc_y \in M_{vy}$, then the events ``$x\notin\uncolvtcs$'', ``$y\notin \uncolvtcs$'', and ``$\phi(x) = c_x$ and $\phi(y) = c_y$'' are not necessarily independent (even if $c_xc_y \notin M_{xy}$), but we assume so in Section~\ref{main proof section}.  By conducting an $\varepsilon$-equalizing coin flip for each vertex and color in its list, this issue is resolved.

  \item The journal version of this paper incorrectly defines pairs and triples to count only nonadjacent pairs and triples.  In the list coloring setting, specifying nonadjacency makes no difference, as any pair of colored neighbors receiving the same color are nonadjacent.  However, in the more general setting of correspondence coloring, the distinction matters.  In particular,~\eqref{random variables save equation} did not hold with the previous definition, and this definition corrects that mistake.  We need to slightly adjust Lemmas~\ref{sparse save} and~\ref{expected savings beats gap} to account for this difference.
  \end{enumerate}
\end{remark}

We are now prepared to state Theorem~\ref{metatheorem}.

\begin{thm}\label{metatheorem}
  For every $\xi_1, \xi_2 > 0$, $\varepsilon,\sigma\in[0, 1)$, and $\rho \in [0,1]$, there exists $\Delta_0$ such that the following holds.  
  If $G$ is a graph with correspondence-assignment $(L, M)$ and a partial ordering $\ordering$ of $V(G)$ such that
  \begin{enumerate}
  \item $\Delta \geq \Delta_0$,
  \item $G$ has maximum degree at most $\Delta$ and minimum degree at least $\delta(\varepsilon)$ (as in Proposition~\ref{keep probability is large}),
  \item[] and for each $v\in V(G)$,
  \item $\Delta \geq |L(v)| \geq (1 - \varepsilon)d(v)$, and
  \item $\Expect{\savings_{v, \sigma, \ordering}} \geq \max\{(1 + \xi_1)\Save_L(v), \xi_2\log^{\logexp}\Delta\}$,
  \end{enumerate}
  then $G$ is $(L, M)$-colorable.
\end{thm}

In order to prove Theorem~\ref{metatheorem}, we need the following lemma.
\begin{lemma}\label{every vertex saves lemma}
  Under the conditions of Theorem~\ref{metatheorem}, if $(\phi, \uncolvtcs)$ is a random naive partial coloring sampled using the local naive random coloring procedure with $\varepsilon$-equalizing coin-flips, then with nonzero probability every $v\in V(G)$ satisfies
  \begin{equation}\label{enough colors equation}
    \Save_{\newList}(v) \leq \unact_{v, \ordering}.
  \end{equation}
\end{lemma}
Observe that by the inclusion-exclusion principle, if we let the \textit{repetitiveness} of color $c\in L(v)$ be one less than the number of colored neighbors $u\in N(v)$ such that $\phi(u)c\in M_{uv}$, then $\pairs_{v, \sigma} - \trips_{v, \sigma}$ undercounts the total repetitiveness of colors assigned to neighbors of $v$.  Therefore
\begin{equation}\label{random variables save equation}
  \Save_L(v) - \Save_{\newList}(v) \geq \aberrance_{v,\sigma} + \pairs_{v,\sigma} - \trips_{v,\sigma}.
\end{equation}

We need to show that with high probability, these random variables are close to their expectation.  We make this precise in the following definition.

\begin{define}
  We say a random variable $X$ is $\Delta$-\textit{concentrated} if
  \begin{equation*}
    \Prob{|X - \Expect{X}| \geq 2\max\{\Expect{X}^{5/6}, \log^{\logexpless}\Delta\}} < \frac{\Delta^{-4}}{16}.
  \end{equation*}
\end{define}

We will use the following lemma to prove Lemma~\ref{every vertex saves lemma}.
\begin{lemma}\label{concentration lemma}
  If $\Delta$ is sufficiently large, $G$ has maximum degree at most $\Delta$, and $\max_v |L(v)| \leq \Delta$, then for each $v\in V(G)$, $\aberrance_{v,\sigma}, \unact_{v, \ordering}, \pairs_{v, \sigma},$ and $\trips_{v, \sigma}$ are $\Delta$-concentrated.
\end{lemma}
We defer the proof of Lemma~\ref{concentration lemma} to Section~\ref{concentration section}.  Lemma~\ref{concentration lemma} is the reason why we need to include the parameter $\sigma$.

To prove Lemma \ref{every vertex saves lemma}, we will also use the Lov\' asz Local Lemma.
\begin{lemma}[Lov\'asz Local Lemma]\label{local lemma}
Let $p\in[0,1)$ and $\mathcal A$ a finite set of events such that for every $A\in\mathcal A$,
\begin{enumerate}
	\item $\Prob{A} \leq p$, and
	\item $A$ is mutually independent of a set of all but at most $d$ other events in $\mathcal A$.
\end{enumerate}
If $4pd\leq 1$, then the probability that none of the events in $\mathcal A$ occur is strictly positive.
\end{lemma}

Now we are ready to prove Lemma~\ref{every vertex saves lemma}.
\begin{proof}[Proof of Lemma~\ref{every vertex saves lemma}]
  For each $v\in V(G)$, let $\mathcal A_v$ be the event that \eqref{enough colors equation} does not hold, and let $\mathcal A = \{\mathcal A_v : v\in V(G)\}$.  Note that for each $v\in V(G)$, $\mathcal A_v$ depends only on trials at vertices at distance at most two from $v$, so if $u\in V(G)$ has distance at least five to $v$, then $\mathcal A_u$ and $\mathcal A_v$ do not depend on any of the same trials.  Therefore each $\mathcal A_v$ is mutually independent of a set of all but at most $\Delta^4$ events in $\mathcal A$.

  By Lemma \ref{local lemma}, it suffices to show that for each $v\in V(G)$, $\Prob{\mathcal A_v} \leq \Delta^{-4}/4$.  Let
  \begin{multline*}
    Z_v = 2(\max\{\Expect{\aberrance_{v,\sigma}}^{5/6}, \log^{\logexpless}\Delta\} + \max\{\Expect{\unact_{v,\ordering}}^{5/6}, \log^{\logexpless}\Delta\}\\
    + \max\{\Expect{\pairs_{v,\sigma}}^{5/6}, \log^{\logexpless}\Delta\} - \max\{\Expect{\trips_{v,\sigma}}^{5/6}, \log^{\logexpless}\Delta\}),
  \end{multline*}
  and let $\mathcal A'_v$ be the event that
  \begin{equation*}
    \savings_{v, \sigma, \ordering} \leq \Expect{\savings_{v, \sigma, \ordering}} - Z_v.
  \end{equation*}
  By Lemma~\ref{concentration lemma} and the Union Bound, $\Prob{\mathcal A'_v} < \Delta^{-4}/4$.  We claim that $\mathcal A_v\subseteq \mathcal A'_v$, which completes the proof.  By \eqref{random variables save equation}, it suffices to show
  \begin{equation}\label{expectation with error beats save}
    \Save_L(v) \leq \Expect{\savings_{v, \sigma, \ordering}} - Z_v.
  \end{equation}
  By the assumption that $\Expect{\savings_{v, \sigma, \ordering}} \geq \xi_2\log^{\logexp}\Delta$, $Z_v = o(\Expect{\savings_{v, \sigma, \ordering}}).$  Since $\Delta$ is sufficiently large, we may assume that $Z_v \leq \xi_1\Save_L(v).$  Since $\Expect{\savings_{v, \sigma, \ordering}} \geq (1 + \xi_1)\Save_L(v)$, \eqref{expectation with error beats save} holds, which completes the proof.
\end{proof}

We conclude this section with the proof of Theorem~\ref{metatheorem}.
\begin{proof}[Proof of Theorem~\ref{metatheorem}]
  By Proposition~\ref{extending partial coloring}, it suffices to show that $G[\uncolvtcs]$ is $(\newList, \newmatching)$-colorable with nonzero probability.  Thus it suffices to show that for some instance of $(\phi, \uncolvtcs)$, for each $v\in \uncolvtcs$,
  \begin{equation}\label{can color greedily equation}
    |\newList(v)| - 1 \geq |\{u\in N(v)\cap \uncolvtcs : u\not\ordering v\}|,
  \end{equation}
  because then we can color $G[\uncolvtcs]$ greedily in the ordering provided by $\ordering$, breaking ties arbitrarily.
  By Lemma \ref{every vertex saves lemma}, we may consider the instance in which each $v\in V(G)$ satisfies \eqref{enough colors equation}.
  
  For each $v\in V(G)$, since $U\supseteq V(G)\setminus A$,
  \begin{equation*}
    |\{u\in N(v)\cap \uncolvtcs : u\not\ordering v\}| = d_\uncolvtcs(v) - |\{u\in N(v)\cap \uncolvtcs : u\ordering v\}| |\geq d_\uncolvtcs(v) - \unact_{v,\ordering}.
  \end{equation*}
  Therefore if \eqref{enough colors equation} holds, then \eqref{can color greedily equation} holds.  Thus, $G[\uncolvtcs]$ is $(\newList, \newmatching)$-colorable, as desired.
\end{proof}
 
\section{Structure}\label{structure section}

The main result of this section is Theorem~\ref{structure thm}, which lower bounds the number of non-adjacent egalitarian neighbors of a vertex in terms of the number of its neighbors that are lordlier, subservient, or weakly egalitarian.

First we need to prove Theorem~\ref{density lemma}, which may be of independent interest.  It bounds the number of edges of a critical graph in terms of the size of a matching in the complement.  Recall that a graph $G$ with list-assignment $L$ is \textit{$L$-critical} if $G$ is not $L$-colorable but every proper induced subgraph of $G$ is.  

\begin{thm}\label{density lemma}
If $G$ is $L$-critical, $H$ is an induced subgraph of $G$, and $M$ is a matching in $\overline H$, then

$$|E(\overline H)| \geq |M|(|V(H)| - |M|) - \sum_{u\in V(H)}\Save_L(u).$$
\end{thm}

We will apply Theorem \ref{density lemma} to an appropriate subset of the neighborhood of each vertex.  In order to prove Theorem~\ref{density lemma}, we need an improved version of a classic result of Erd\H os, Rubin, and Taylor \cite{ERT79} about list-coloring a complete graph with a matching removed, proved by Delcourt and Postle~\cite{DP17}.  We include a proof for completeness.

\begin{lemma}[Delcourt and Postle \cite{DP17}]\label{coloring K_n - M}
If $G = K_n - M$, where $M$ is a matching and $L$ is a list-assignment for $G$ such that
\begin{enumerate}
	\item for all $ab\in M$, $|L(a)|, |L(b)|\geq |M|$ and $|L(a)| + |L(b)| \geq n$,
	\item for all $v\in V(G-V(M))$, $|L(v)|\geq n - |M|$,
\end{enumerate}
then $G$ is $L$-colorable.
\end{lemma}
\begin{proof}
We proceed by induction on $n$.  If $n\leq 1$, then $M=\varnothing$ and by 2, $|L(v)|\geq n$ for all $v\in G$.  So we may assume $n\geq 2$.

Suppose there exists $ab\in M$ such that $L(a)\cap L(b)\neq\varnothing$.
Let $c\in L(a)\cap L(b)$, and for all $v\in V(G)\backslash\{a,b\}$, let $L'(v) = L(v)\backslash \{c\}$.
Let $G' = G-a-b$ and $M' = M-ab$.
Then $G',M',$ and $L'$ satisfy conditions 1 and 2.
By induction, $G'$ has an $L'$-coloring.
Therefore $G$ has an $L$-coloring, obtained from an $L'$-coloring $G'$ by coloring $a$ and $b$ with color $c$, as desired.

Therefore we may assume that for all $ab\in M$, $L(a)\cap L(b)=\varnothing$.
Since $|L(a)| + |L(b)| \geq n$, $|L(a)\cup L(b)| \geq n$.
We claim for all $X\subseteq V(G)$, $|\bigcup_{v\in X}L(v)| \geq |X|$.
If there exists $ab\in M$ such that $a,b\in X$, then $|\bigcup_{v\in X}L(v)| \geq n \geq |X|$, as claimed.
Therefore we may assume that $|X| \leq n - |M|$.
If $X\backslash V(M)\neq\varnothing$, then $|\bigcup_{v\in X}L(v)| \geq n - |M| \geq |X|$, as claimed.
Hence, we may assume that $X\subseteq V(M)$.
But then $|\bigcup_{v\in X}L(v)| \geq |M| \geq |X|$, as claimed.  

Therefore $|\bigcup_{v\in X}L(v)| \geq |X|$ for all $X\subseteq V(G)$.  By Hall's Theorem, there is a matching from $V(G)$ to $\cup_v L(v)$, and thus $G$ has an $L$-coloring, as desired.
\end{proof}

Now we prove Theorem \ref{density lemma}.

\begin{proof}[Proof of Theorem \ref{density lemma}]
We proceed by induction on $|V(H)|$.  Since $G$ is $L$-critical, $G-V(H)$ has an $L$-coloring $\phi$.  For all $v\in V(H)$, let $L'(v) = L(v) \backslash \{\phi(u) : u\in N(v)\cap V(G-V(H))\}$.  Since $G$ does not have an $L$-coloring, $H$ does not have an $L'$-coloring.  By Lemma \ref{coloring K_n - M}, either there exists $ab\in M$ such that $|L(a)| < |M|$ or $|L(a)| + |L(b)| < |V(H)|$ or there exists $v\in V(H-V(M))$ such that $|L'(v)| < |V(H)| - |M|$.  Note that for all $v\in V(H)$,
\begin{equation}\label{new list size}
\begin{split}
|L'(v)| &\geq |L(v)| - d_{G-V(H)}(v)\\
 &= d_{H}(v) + 1 - \Save_L(v).
 \end{split}
 \end{equation}

If there exists $ab\in M$ such that $|L'(a)| < |M|$, then let $H' = H-a$ and $M' = M - ab$.  By \eqref{new list size}, $d_{H}(a) + 1 - \Save_L(a) < |M|$.  Hence, 
\begin{align*}
d_{\overline H}(a) &= |V(H)| - 1 - d_{H}(a)\\
&> |V(H)| - |M| - \Save_L(a).
\end{align*}
By induction, $|E(\overline H')| \geq |M'|(|V(H')| - |M'|) - \sum_{u\in V(H')}\Save_L(u)$.  Therefore,
\begin{align*}
|E(\overline H)| &= |E(\overline H')| + d_{\overline H}(a)\\
 &> |M'|(|V(H')| - |M'|) - \sum_{u\in V(H')}\Save_L(u) + |V(H)| - |M| - \Save_L(a)\\
 &= (|M| - 1)(|V(H)| - |M|) + |V(H)| - |M| - \sum_{u\in V(H)}\Save_L(u)\\
 &=|M|(|V(H)| - |M|) - \sum_{u\in V(H)}\Save_L(u),
\end{align*}
as desired.

If there exists $ab\in M$ such that $|L'(a)| + |L'(b)| < |V(H)|$ then let $H'=H - a - b$ and $M' = M - ab$.  By \eqref{new list size}, $d_{H}(a) + 1 - \Save_L(a) + d_{H}(b) + 1 - \Save_L(b) < |V(H)|$.  Hence, 
\begin{align*}
|\delta_{\overline{H}}(\{a,b\})| &= 2(|V(H)| - 2) - d_{H}(a) - d_{H}(b)\\
&> 2(|V(H)| - 2) - |V(H)| + 2 - \Save_L(a) - \Save_L(b)\\
&= |V(H)| - 2 - \Save_L(a) - \Save_L(b),
\end{align*}
where $\delta_{\overline{H}}(\{a, b\})$ is the set of edges in $\overline H$ incident to precisely one of $a$ and $b$.

By induction, $|E(\overline H')| \geq |M'|(|V(H')| - |M'|) - \sum_{u\in V(H')}\Save_L(u)$.  Therefore,
\begin{align*}
|E(\overline H)| &= |E(\overline H')| + \delta_{\overline H}(\{a,b\}) + 1\\
 &\geq |M'|(|V(H')| - |M'|) - \sum_{u\in V(H')}\Save_L(u) + |V(H)| - \Save_L(a) - \Save_L(b)\\
 &= (|M| - 1)(|V(H)| - |M| - 1) + |V(H)| - \sum_{u\in V(H)}\Save_L(u)\\
 &>|M|(|V(H)| - |M|) - \sum_{u\in V(H)}\Save_L(u),
\end{align*}
as desired.

Otherwise, there exists some $v\in V(H-V(M))$ such that $|L'(v)| < |V(H)| - |M|$, so let $H' = H - v$.  By \eqref{new list size}, $d_{H'}(v) + 1 - \Save_L(v) < |V(H)| - |M|$.  Hence,
\begin{align*}
d_{\overline H}(v) &= |V(H)| - 1 - d_{H}(v)\\
&> |M| - \Save_L(v).
\end{align*}
By induction, $|E(\overline H')| \geq |M|(|V(H')| - |M|) - \sum_{u\in V(H')}\Save_L(u)$.  Therefore,
\begin{align*}
|E(\overline H)| &= |E(\overline H')| + d_{\overline H}(v)\\
 &> |M|(|V(H')| - |M|) - \sum_{u\in V(H')}\Save_L(u) + |M| - \Save_L(v)\\
 &= |M|(|V(H)| - |M| - 1) + |M| - \sum_{u\in V(H)}\Save_L(u)\\
 &=|M|(|V(H)| - |M|) - \sum_{u\in V(H)}\Save_L(u),
\end{align*}
as desired.
\end{proof}

Recall that if $v$ is a vertex of a graph $G$ and $u\in N(v)$, we say $u$ is a subservient neighbor of $v$ if $|L(u)| < |L(v)|$, a strongly egalitarian neighbor of $v$ if $|L(u)|\in [|L(v)|, |L(v)| + \beta\Gap(v))$, a weakly egalitarian neighbor of $v$ if $|L(u)| \in [|L(v)| + \beta\Gap(v), (1 + \alpha)|L(v)|)$, and a lordlier neighbor of $v$ if $|L(u)| \geq (1 + \alpha)|L(v)|$.  Recall also that this partitions the neighbors of $v$ into the sets $\subserv(v), \strongEgal(v), \weakEgal(v),$ and $\aberrant(v)$, the sets of subservient, strongly egalitarian, weakly egalitarian, and lordlier neighbors of $v$, respectively, and that we let $\notEgal(v) = N(v) - \egal(v)$.

The following is the main result of this section.
\begin{thm}\label{structure thm}
  Let $\varepsilon\in (0, 1)$.  If $G$ is an $L$-critical graph for some list-assignment $L$ such that for every $v\in V(G)$, $|L(v)|\geq \varepsilon\omega(v) + (1 - \varepsilon)(d(v) + 1)$, then for all $v\in V(G)$,
  \begin{multline*}
  |E(\overline{G[\egal(v)]})| \geq \left(\frac{1}{4} - \frac{\varepsilon(4 + \beta + 2\alpha)}{2(1 - \varepsilon)}\right)\Gap(v)d(v) - \left(\frac{1}{2} - \frac{\varepsilon(1 + \beta)}{2(1 - \varepsilon)}\right)d(v)|\notEgal(v)| \\
  - \left(\frac{1}{4} - \frac{\varepsilon(2 + \beta)}{2(1 - \varepsilon)}\right)\Gap(v)|\weakEgal(v)|.
\end{multline*}
 \end{thm}

For the remainder of this section, we assume that $G$ is a graph with list-assignment $L$ satisfying the conditions of Theorem~\ref{structure thm}.

Theorem~\ref{structure thm} is useful because it implies that if $v$ does not have many lordlier or subservient neighbors, or many weakly egalitarian neighbors, then it has many non-adjacent egalitarian neighbors.  We prove Theorem~\ref{structure thm} by considering a maximum antimatching $M$ among $v$'s egalitarian neighbors and applying Theorem~\ref{density lemma} with $H = G[V(M)\cup \strongEgal(v)]$.  If $u$ is a strongly egalitarian neighbor of $v$, then $\Save_L(u)$ is close to $\Save_L(v)$.  If $u$ is a weakly egalitarian neighbor of $v$, then we can not bound $\Save_L(u)$ well enough, so we do not include $u$ in $H$ unless $u$ is in the antimatching.

We will use the following propositions to prove Theorem \ref{structure thm}.  First, we need to bound the size of a maximum antimatching taken among the egalitarian neighbors, as in the following proposition.

\begin{prop}\label{matching bound} 
If $M$ is a maximum matching in $\overline{G[\egal(v)]}$, then
\begin{equation*}
\frac{\Gap(v) - |\notEgal(v)|}{2} \leq |M| \leq \Gap(v).
\end{equation*}
\end{prop}
\begin{proof}
Since $M$ is maximum, $G[\egal(v) - V(M)]$ is a clique, so $2|M| \geq |\egal(v)| - \omega(G[\egal(v)]) \geq |\egal(v)| - \omega(v) = \Gap(v) - |\notEgal(v)|$, as desired.

Since no clique in $G[\egal(v)]$ contains an edge in $M$, $\omega(G[\egal(v)]) \leq |\egal(v)| - |M|$.  Note that for any $H\subseteq G[N(v)\cup\{v\}]$, $|V(H)| - \omega(H) \leq \Gap(v)$.  Hence, $|M| \leq \Gap(v)$, as desired.
\end{proof}

\begin{prop}\label{gap bound for egal}
If $u$ is an egalitarian neighbor of a vertex $v$ (i.e.\ $u\in \egal(v)$), then
\begin{equation*}
\Gap(u) \leq \frac{1 + \alpha}{1 - \varepsilon}d(v).
\end{equation*}
\end{prop}
\begin{proof}
Since $G$ is $L$-critical, $|L(v)| \leq d(v)$.  Since $u\in \egal(v)$, $|L(u)| \leq (1 + \alpha)|L(v)|$.  Hence, $|L(u)| \leq (1 + \alpha)d(v)$.  Since $|L(u)| \geq (1 - \varepsilon)d(u)$, $d(u) \leq \frac{1 + \alpha}{1 - \varepsilon}d(v)$.  Since $\Gap(u) \leq d(u)$, the result follows.
\end{proof}

Since we will apply Theorem~\ref{density lemma}, we will need to upper bound $\Save_L(u)$ for egalitarian neighbors $u$ of $v$.  Since $\Save_L(u) \leq \varepsilon\Gap(u)$, it suffices to upper bound $\Gap(u)$.
Proposition~\ref{gap bound for egal} provides a rough bound on $\Gap(u)$ that we will use for the egalitarian neighbors in the antimatching.  The next proposition provides an improved bound on $\Gap(u)$ if $u$ is a strongly egalitarian neighbor that is not in the antimatching.

\begin{prop}\label{gap bound for C}
If $M$ is a maximum matching in $\overline{G[\egal(v)]}$ and $u\in \strongEgal(v) - V(M)$, then
\begin{equation*}
\Gap(u) \leq \frac{(2 + \beta)\Gap(v) + |\notEgal(v)|}{1 - \varepsilon}.
\end{equation*}
\end{prop}
\begin{proof}
  Since $M$ is maximum, $G[\egal(v) - V(M)]$ is a clique, so
  \begin{equation}
    \label{omega bound for u in C}
    \omega(u) \geq |\egal(v)| - 2|M|.
  \end{equation}
  Since $u\in \strongEgal(v)$, $|L(u)| \leq |L(v)| + \beta \Gap(v)$.  Since $G$ is $L$-critical, $|L(v)| \leq d(v)$.  Hence, $|L(u)| \leq d(v) + \beta\Gap(v)$.  Since $d(u) \leq \frac{|L(u)| - \varepsilon\omega(u)}{1 - \varepsilon}$, 
  \begin{equation}
    \label{degree bound for u in C}
    d(u) \leq \frac{d(v) + \beta\Gap(v) - \varepsilon\omega(u)}{1 - \varepsilon}.
  \end{equation}
  Now the result follows from \eqref{omega bound for u in C}, \eqref{degree bound for u in C}, and Proposition~\ref{matching bound}.
\end{proof}

Now we are ready to prove Theorem \ref{structure thm}.
\begin{proof}[Proof of Theorem \ref{structure thm}]
Let $M$ be a maximum matching in $\overline{G[\egal(v)])}$, and let $\weakEgal'(v) = \weakEgal(v) - V(M)$.  
Let $H = G[V(M)\cup \strongEgal(v)])$. 
By Theorem \ref{density lemma},
\begin{equation}\label{applying density lemma}
|E(\overline{H})| \geq |M|(|V(H)| - |M|) - \sum_{u\in V(H)} \Save(u).
\end{equation}
 By Proposition \ref{gap bound for C}, 
 \begin{multline}
   \label{savings not in matching}
  \sum_{u\in V(H-V(M))}\Save(u) \leq \sum_{u\in V(H-V(M))}\varepsilon\Gap(u)\\
  \leq (|V(H)|-|M|)\left(\frac{\varepsilon}{1 - \varepsilon}\right)((2 + \beta)\Gap(v) + |\notEgal(v)|).
\end{multline}
By Proposition \ref{gap bound for egal} and \ref{matching bound}.
\begin{equation}
  \label{savings in matching}
  \sum_{u\in V(M)}\Save(u) \leq \frac{\varepsilon(1 + \alpha)}{1 - \varepsilon}d(v)|M| \leq \frac{\varepsilon(1 + \alpha)\Gap(v)d(v)}{1 - \varepsilon}.
\end{equation}

By \eqref{applying density lemma}, \eqref{savings not in matching}, and \eqref{savings in matching},
\begin{equation}
  \label{savings with V - M}
  |E(\overline H)| \geq (|V(H)| - |M|)\left(|M| - \frac{\varepsilon((2 + \beta)\Gap(v) + |\notEgal(v)|)}{1 - \varepsilon}\right) - \frac{\varepsilon(1 + \alpha)\Gap(v)d(v)}{1 - \varepsilon}.
\end{equation}
Note that $|M| \leq |V(H)|/2$, so $|V(H)| - |M| \geq |V(H)|/2$.  Therefore by Proposition \ref{matching bound} and \eqref{savings with V - M},
\begin{multline}
  \label{savings with V - M rearranged}
  |E(\overline H)| \geq \left(\frac{|V(H)|}{2}\right)\left(\Gap(v)\left(\frac{1}{2} - \frac{\varepsilon(2 + \beta)}{1 - \varepsilon}\right) - |\notEgal(v)|\left(\frac{1}{2} + \frac{\varepsilon}{1 - \varepsilon}\right)\right)\\
     - \frac{\varepsilon}{1 - \varepsilon}(1 + \alpha)\Gap(v)d(v).
\end{multline}

Since $|V(H)| = d(v) - |\notEgal(v)| - |\weakEgal'(v)|$, by combining terms in~\eqref{savings with V - M rearranged} and ignoring some positive terms, we have that
\begin{multline*}
  |E(\overline H)| \geq \Gap(v)d(v)\left(\frac{1}{4} - \frac{\varepsilon(4 + \beta + 2\alpha)}{2(1 - \varepsilon)}\right) - d(v)|\notEgal(v)|\left(\frac{1}{4} + \frac{\varepsilon}{2(1 - \varepsilon)}\right)\\
  - \Gap(v)|\notEgal(v)|\left(\frac{1}{4} - \frac{\varepsilon(2 + \beta)}{2(1 - \varepsilon)}\right) - \Gap(v)|\weakEgal'(v)|\left(\frac{1}{4} - \frac{\varepsilon(2 + \beta)}{2(1 - \varepsilon)}\right).
\end{multline*}
Since $\Gap(v)\leq d(v)$, $|\weakEgal'(v)| \leq |\weakEgal(v)|$, and $|E(\overline{G[\egal(v)]})| \geq |E(\overline H)|$,
\begin{multline*}
  |E(\overline{G[\egal(v)]})| \geq \left(\frac{1}{4} - \frac{\varepsilon(4 + \beta + 2\alpha)}{2(1 - \varepsilon)}\right)\Gap(v)d(v) - \left(\frac{1}{2} - \frac{\varepsilon(1 + \beta)}{2(1 - \varepsilon)}\right)d(v)|\notEgal(v)| \\
  - \left(\frac{1}{4} - \frac{\varepsilon(2 + \beta)}{2(1 - \varepsilon)}\right)\Gap(v)|\weakEgal(v)|,
\end{multline*}
as desired.                  
\end{proof}

Note that we could take a maximum antimatching among the strongly egalitarian neighbors of a vertex $v$ and follow the same proof strategy of Theorem~\ref{structure thm} to obtain a bound of
\begin{equation*}
  |E(\overline{G[\strongEgal(v)]})| \geq \Omega(\Gap(v) d(v)) - O(d(v))|N(v) - \strongEgal(v)|.
\end{equation*}
However, this is not a good enough bound, because if there are $\Omega(\Gap(v))$ weakly egalitarian neighbors of $v$, we do not have enough non-adjacent strongly egalitarian neighbors to expect many colors assigned to multiple neighbors of $v$, and we do not expect enough weakly egalitarian neighbors to receive a color not in $L(v)$.
 
\section{Proof of Theorem~\ref{main thm}}\label{main proof section}

In this section, we prove Theorem~\ref{main thm}.
In order for the proof of Theorem \ref{main thm} to work inductively, we actually prove the following.
\begin{thm}\label{inductive thm}
Let $\varepsilon = \frac{1}{330}$. There exists $\Delta_0$ such that for all $\Delta \geq \Delta_0$, if $G$ is a graph of maximum degree at most $\Delta$ with a list assignment $L$ such that for all $v\in V(G)$,
\begin{enumerate}
	\item $|L(v)| \geq \omega(v) + \log^{\logexp}(\Delta)$ and
	\item $|L(v)| \geq \epsilon\omega(v) + (1 - \epsilon)(d(v) + 1)$,
\end{enumerate} 
then $G$ is $L$-colorable.
\end{thm}

Note that Theorem~\ref{main thm} follows immediately from Theorem~\ref{inductive thm}.  If $G$ is a graph with list-assignment $L$ that satisfies the conditions of Theorem~\ref{inductive thm}, then so does any subgraph of $G$.  Therefore in proving Theorem~\ref{inductive thm}, we may assume $G$ is $L$-critical, and hence we can apply Theorem~\ref{structure thm}.

For the remainder of this section, unless specified otherwise, $G, L, \varepsilon, \rho$, and $\Delta$ are assumed to satisfy the conditions of Theorem \ref{inductive thm}, and we assume that $G$ is $L$-critical.  For each edge $e\in E(G)$, we let $M_e$ be a matching of $\{u\}\times L(u)$ and $\{v\}\times L(v)$ such that $(L, M)$ is a total correspondence assignment for $G$ where every $(L, M)$-coloring of $G$ is an $L$-coloring.  Let $(\phi, \uncolvtcs)$ be a random naive partial coloring and $A$ a set of activated vertices sampled using the local naive random coloring procedure with activation probability $\rho$ and $\varepsilon$-equalizing coin flips.  Note that we are assuming $G$ is $L$-critical before assuming the correspondence assignment is total, since Theorem~\ref{structure thm} does not hold for correspondence coloring.  Recall that we let $K_{\varepsilon, \rho} = .999e^{\frac{-\rho}{1 - \varepsilon}}$.  For convenience, let $K = K_{\varepsilon, \rho}$.  For $u,v\in V(G)$, let $u\ordering v$ if $|L(u)| < |L(v)|$.

Before proving Theorem~\ref{main thm}, we need to lower bound the expected savings for each vertex, as in the following lemmas.
\begin{lemma}\label{lordly save}
  For each $v\in V(G)$,
  \begin{equation*}
    \Expect{\unact_{v,\ordering}} = (1 - \rho)|\subserv(v)|.
  \end{equation*}
\end{lemma}
\begin{proof}
  Since each neighbor of $v$ is in $A$ with probability $\rho$, the lemma follows by linearity of expectation.
\end{proof}

\begin{lemma}\label{aberrant save}
  For each $v\in V(G)$,
  \begin{equation*}
    \Expect{\aberrance_{v,0}} \geq K\left(\frac{\alpha}{1 + \alpha}|\aberrant(v)| + \frac{\beta\Gap(v)}{d(v) + \beta\Gap(v)}|\weakEgal(v)|\right).
  \end{equation*}
\end{lemma}
\begin{proof}
  Let
  \begin{equation*}
    \totesabs_v = |\{u \in \egal(v) : \phi(u)\notin V(M_{uv})\}|,
  \end{equation*}
  and note that $\Expect{\aberrance_{v,0}} = K\cdot\Expect{\totesabs_v}$ by Proposition~\ref{equalized keep probability}.
  For each $u\in \aberrant(v)$,
  \begin{equation*}
    \Prob{\phi(u)\notin V(M_{uv})} \geq \frac{\alpha}{1 + \alpha},
  \end{equation*}
  and for each $u\in \weakEgal(v)$,
  \begin{equation*}
    \Prob{\phi(u)\notin V(M_{uv})} \geq \frac{\beta\Gap(v)}{|L(v)| + \beta\Gap(v)} \geq \frac{\beta\Gap(v)}{d(v) + \beta\Gap(v)}.
  \end{equation*}
  Therefore it follows that
  \begin{equation*}
    \Expect{\totesabs_v} \geq \frac{\alpha}{1 + \alpha}|\aberrant(v)| + \frac{\beta\Gap(v)}{d(v) + \beta\Gap(v)}|\weakEgal(v)|.
  \end{equation*}
  Since $\Expect{\aberrance_{v,0}} = K\cdot\Expect{\totesabs_v}$, the result follows.
\end{proof}

Recall that we apply Theorem~\ref{metatheorem} with $\sigma = 0$.  Thus we need to bound $\pairs_{v, 0} - \trips_{v, 0}$, as in the following lemma.
\begin{lemma}\label{sparse save}
  For each $v\in V(G)$,
  \begin{equation*}
    \Expect{\pairs_{v,0} - \trips_{v,0}} \geq \min_{i\in\{1,2\}}\left(\frac{K\cdot e_i}{|L(v)|}\right)\left(\frac{K}{(1 + \alpha)^2} - \frac{\sqrt{2\cdot e_i}}{3|L(v)|}\right),
  \end{equation*}
  where $e_1 = |E(\overline{G[\egal(v)]})|$ and $e_2 = \binom{d(v)}{2}$.
\end{lemma}
\begin{proof}
  We let $T(H)$ denote the set of triangles in a graph $H$.
  We define the following random variables for each $c\in L(v)$:
  \begin{align*}
    &\totespairs_{v,c} = |\{x, y \in \egal(v): \phi(x)c\in M_{xv}\text{ and }\phi(y)c\in M_{yv}\}|, \text{\ and}\\
    &\totestrips_{v,c} = |\{x, y, z\in \egal(v): \phi(x)c\in M_{xv}, \phi(y)c\in M_{yv},\text{ and }\phi(z)c\in M_{zv}\}|,
  \end{align*}
  and we define $\totespairs_v = \sum_{c\in L(v)}\totespairs_{v,c}$ and $\totestrips_v = \sum_{c\in L(v)}\totestrips_{v,c}$.

  For each $c\in L(v)$, let $H_c$ be the subgraph of $G[\egal(v)]$ defined as follows.  A vertex $x\in \egal(v)$ is in $V(H_c)$ if there exists a color $c_x\in L(x)$ such that $cc_x \in M_{vx}$, and $xy \in E(H_c)$ if $xy\in E(G)$ and moreover $c_xc_y \in M_{xy}$, where $cc_x \in M_{vx}$ and $cc_y \in M_{vy}$.  For any pair $xy \in E(\overline H_c)$, we have $\ProbCond{x, y \notin U}{\phi(x) = c_x,~\phi(y) = c_y} \geq K^2$.  Moreover, for any triple $xyz\in T(\overline H_c)$, we have $\ProbCond{x, y, z\notin U}{\phi(w) = c_w~\forall~w\in\{x,y,z\}} \leq \Prob {x \notin U} \leq K$.  Hence,
  \begin{equation}\label{pairs-trips-lower-bound-by-total}
    \Expect{\pairs_{v,0} - \trips_{v,0}} \geq K^2\Expect{\totespairs_v} - K\Expect{\totestrips_v}.
  \end{equation}

  In 2002, Rivin \cite{R02} proved that
  \begin{equation}
    \label{triangle-bound}
    |T(H)| \leq \frac{(2|E(H)|)^{\frac{3}{2}}}{6}.
  \end{equation}
    For every $c\in L(v)$, we have $V(H_c) = \egal(v)$ since $(L, M)$ is total, so
  \begin{equation*}
    \Expect{\totespairs_{v,c}}  = \sum_{xy\in E(\overline {H_c})}\frac{1}{|L(x)||L(y)|}.
  \end{equation*}
  By the definition of $\egal(v)$, if $x,y\in \egal(v)$,
  \begin{equation*}
    \frac{1}{|L(x)||L(y)|} \geq \frac{1}{(1 + \alpha)^2|L(v)|^2}. 
  \end{equation*}
  Therefore
  \begin{equation}
    \label{total pairs expectation}
    \Expect{\totespairs_{v,c}} \geq \frac{|E(\overline{H_c})|}{(1 + \alpha)^2|L(v)|^2}.
  \end{equation}
  Similarly,
  \begin{equation*}
    \Expect{\totestrips_{v,c}} = \sum_{xyz\in T(\overline{H_c})}\frac{1}{|L(x)||L(y)||L(z)|},
  \end{equation*}
  and
  \begin{equation*}
    \frac{1}{|L(x)||L(y)||L(z)|} \leq \frac{1}{|L(v)|^3}. 
  \end{equation*}
  Therefore
  \begin{equation}
    \label{trips expectation with triangles}
    \Expect{\totestrips_{v,c}} \leq \frac{|T(\overline{H_c})|}{|L(v)|^3}.
  \end{equation}
  By \eqref{triangle-bound} and \eqref{trips expectation with triangles},
  \begin{equation}
    \label{total trips expectation}
    \Expect{\totestrips_{v,c}} \leq \frac{\sqrt{8}|E(\overline{H_c})|^{3/2}}{6|L(v)|^3}.
  \end{equation}
  It follows from \eqref{pairs-trips-lower-bound-by-total}, \eqref{total pairs expectation}, and \eqref{total trips expectation} that
  \begin{equation}\label{pairs-triples-bound-per-color}
    \Expect{\pairs_v - \trips_v} \geq \sum_{c\in L(v)}\left(\frac{K|E(\overline{H_c})|}{|L(v)|^2}\right)\left(\frac{K}{(1 + \alpha)^2} - \frac{\sqrt{2|E(\overline{H_c})|}}{3|L(v)|}\right).
  \end{equation}

  Since $H_c \subseteq G[\egal(v)]$ for every $c\in L(v)$, we have $|E(\overline{G[\egal(v)]})| \leq |E(\overline{H_c})| \leq \binom{d(v)}{2}$.
  For any constants $a$ and $b$, the function $(a - b\sqrt{x})x$ is increasing for $0 \leq x < (2a/(3b))^2$ and decreasing for $x > (2a/(3b))^2$.  Letting $a = K / (1 + \alpha)^2$ and $b = \sqrt{2}/(3|L(v)|)$, this fact implies that each term in the sum in the right side of~\eqref{pairs-triples-bound-per-color} is at least as large as the minimum of two values: the value of the term when $|E(\overline{H_c})|$ is either $|E(\overline{G[\egal(v)]})|$ or simply $\binom{d(v)}{2}$.  Since there are $|L(v)|$ terms in the sum, the result follows.
\end{proof}

Combining Theorem~\ref{structure thm} with Lemmas~\ref{lordly save}, \ref{aberrant save}, and \ref{sparse save}, we prove that the expected savings for each vertex $v$ is larger than $\varepsilon\Gap(v)$, as follows.
\begin{lemma}\label{expected savings beats gap}
  Let $\alpha = \beta = \frac{1}{50}$ and $\rho = 1 - e^{-1}\alpha / (1 + \alpha)$.  For each vertex $v\in V(G)$,
  \begin{equation*}
    \Expect{\savings_{v, 0, \ordering}} \geq 1.01\varepsilon\Gap(v).
  \end{equation*}
\end{lemma}

\begin{proof}    
  By Theorem \ref{structure thm},
  \begin{multline}
    \label{egal sparsity lower bound}
    |E(\overline{G[\egal(v)]})| \geq \left(\frac{1}{4} - \frac{\varepsilon(4 + \beta + 2\alpha)}{2(1 - \varepsilon)}\right)\Gap(v)d(v) - \left(\frac{1}{2} - \frac{\varepsilon(1 + \beta)}{2(1 - \varepsilon)}\right)d(v)|\notEgal(v)| \\
    - \left(\frac{1}{4} - \frac{\varepsilon(2 + \beta)}{2(1 - \varepsilon)}\right)\Gap(v)|\weakEgal(v)|.
  \end{multline}

  By Lemmas~\ref{lordly save} and~\ref{aberrant save}, we may assume that
  \begin{equation*}
    (1 - \rho)|\subserv(v)| + \frac{K\alpha}{1 + \alpha}\left(|\aberrant(v)| + \frac{\Gap(v)}{d(v)}|\weakEgal(v)|\right) \leq 1.01\varepsilon\cdot\Gap(v).
  \end{equation*}
  Subject to this inequality, since $1 - \rho  \geq \frac{K\alpha}{1 + \alpha}$ and $\frac{1}{2} - \frac{\varepsilon(1 + \beta)}{2(1 - \varepsilon)} \geq \frac{1}{4} - \frac{\varepsilon(2 + \beta)}{2(1 - \varepsilon)}$, the right side of~\eqref{egal sparsity lower bound} is at least as large as the case when $|\subserv(v)| = |\weakEgal(v)| = 0$ and $|\aberrant(v)| \leq 1.01\varepsilon(1 + \alpha)\Gap(v)/(K\alpha)$, that is
  \begin{equation*}
    |E(\overline{G[\egal(v)]})| \geq \Gap(v)d(v)\left(\frac{1}{4} - \frac{\varepsilon(4 + \beta + 2\alpha)}{2(1 - \varepsilon)} - 1.01\varepsilon\frac{1 + \alpha}{\alpha K}\left(\frac{1}{2} - \frac{\varepsilon(1 + \beta)}{2(1 - \varepsilon)}\right)\right).
  \end{equation*}
  Therefore, since $\Gap(v),|L(v)| \leq d(v) \leq |L(v)|/(1 - \varepsilon)$, by Lemma~\ref{sparse save}, we have
  \begin{equation}\label{total expected savings equation}
    \Expect{\savings_{v, 0, \ordering}} \geq \min_{i\in\{1,2\}} K\left(\frac{K}{(1 + \alpha)^2} - \frac{\left(2\cdot\mathrm{sparsity}_i(\alpha, \beta, \varepsilon)\right)^{1/2}}{3(1 - \varepsilon)}\right)\Gap(v)\cdot \mathrm{sparsity}_i(\alpha, \beta, \varepsilon),
  \end{equation}
  where $\mathrm{sparsity}_1(\alpha, \beta, \varepsilon) = \frac{1}{4} - \frac{\varepsilon(4 + \beta + 2\alpha)}{2(1 - \varepsilon)} - 1.01\varepsilon\frac{1 + \alpha}{\alpha K}\left(\frac{1}{2} - \frac{\varepsilon(1 + \beta)}{2(1 - \varepsilon)}\right)$ and $\mathrm{sparsity}_2(\alpha, \beta, \varepsilon) = 1/2$.

  Since $\alpha = \beta = 1/50$, $\varepsilon = 1/330$, and $K = .999\rho e^{-330\rho/329}$, the right side of \eqref{total expected savings equation} is at least  $1.01\varepsilon\Gap(v)$, as required.
\end{proof}

Finally we can prove Theorem~\ref{main thm}.
\begin{proof}[Proof of Theorem~\ref{main thm}]
  Actually we prove Theorem \ref{inductive thm}.
  Recall that we assume $G$ is $L$-critical, and we assume $(L, M)$ is a total correspondence assignment for $G$ such that an $(L, M)$-coloring is an $L$-coloring.
  We will apply Theorem~\ref{metatheorem} with $\xi_1 = 1.01$, $\xi_2 = \varepsilon$, $\rho = 1 - 1/(50e(1 + 1/50))$, and $\sigma = 0$ to show that $G$ is $(L, M)$-colorable, contradicting that $G$ is $L$-critical.
  
  Let $v\in V(G)$.
  Since $G$ is $L$-critical, $d(v) \geq |L(v)| \geq \log^{\logexp}(\Delta)$.  Hence we may assume that $G$ has minimum degree at least $\delta(\varepsilon)$.
  By Lemma~\ref{expected savings beats gap}, since $\Gap(v) - \Save(v) \geq \log^{\logexp}(\Delta)$,
  \begin{equation*}
    \Expect{\savings_{v, 0, \ordering}} \geq \xi_2\log^{\logexp}\Delta,
  \end{equation*}
  and since $\Save(v) \leq \varepsilon\Gap(v)$, 
  \begin{equation*}
    \Expect{\savings_{v, 0, \ordering}} \geq \xi_1\Save(v).
  \end{equation*}
  Therefore by Theorem~\ref{metatheorem}, $G$ is $(L, M)$-colorable, a contradiction.
\end{proof}
 
\section{Concentrations}\label{concentration section}

In this section we prove Lemma~\ref{concentration lemma}.  Recall that $\varepsilon, \sigma\in[0, 1)$, $\rho\in[0,1]$, $G$ is a graph with correspondence-assignment $(L, M)$ satisfying the assumptions of Definition~\ref{random coloring procedure}, $G$ has maximum degree at most $\Delta$, $\max_v |L(v)| \leq \Delta$, and $\Delta$ is sufficiently large.

We prove Lemma~\ref{concentration lemma} using Talagrand's Inequality.  Instead of applying Talagrand's Inequality in its original form (see Theorem~\ref{OG Tala}), it is common to derive from it a ``concentration inequality''.  The following theorem is such an example; it appears in the book of Molloy and Reed~\cite[Chapter 10]{MR00}.

\begin{thm}[``Talagrand's Inequality II'' \cite{MR00}]\label{mr tala}
  Let $X$ be a non-negative random variable, not identically 0, which is determined by $n$ independent trials $T_1, \dots, T_n$, and satisfying the following for some $c, r > 0$:
  \begin{enumerate}
  \item changing the outcome of any one trial can affect $X$ by at most $c$, and
  \item for any $s$, if $X\geq s$ then there is a set of at most $rs$ trials whose outcomes certify that $X\geq s$,
  \end{enumerate}
  then for any $0 \leq t \leq \Expect{X}$,
  \begin{equation*}
    \Prob{|X - \Expect{X}| > t + 60c\sqrt{r\Expect{X}}} \leq 4\exp\left(-\frac{t^2}{8c^2r\Expect{X}}\right).
  \end{equation*}
\end{thm}

\begin{remark}\label{talagrand mistake remark}
  As we explain in Appendix~\ref{tala proof section}, Molloy and Reed's~\cite{MR00} proof of Theorem~\ref{mr tala} is flawed, and we correct this flaw.  After submitting the initial version of this paper, we discovered that Molloy and Reed~\cite{MR14} later published a corrected version of Theorem~\ref{mr tala}, but they did not mention that the initial version is incorrect nor did they explain the flaw.  We discuss this correction further later in this section in Remark~\ref{talagrand correction remark} and also in Appendix~\ref{tala proof section}.  We remark that all applications of Theorem~\ref{mr tala} that we know of follow from the version in~\cite{MR14} as well as Theorem~\ref{exceptional talagrand's}.
\end{remark}

Unfortunately Theorem~\ref{mr tala} is too restrictive for us.  In our situation, changing the outcome of a trial is changing either whether or not a vertex is activated, the color assigned to a vertex, or the outcome of an ``equalizing coin flip,'' and this affects the value of our random variables if many vertices are randomly assigned the same color.  For example, if $G$ is $\Delta$-regular and
$v$ is a vertex, then
we may sample a naive partial coloring $(\phi, \uncolvtcs)$ such that $u \notin U$ and $\phi(u)\notin V(M_{uv})$ for each $u\in N(v)$.  In this case, the value of $\aberrance_{v, \sigma}$ is $\Delta$; however, if there is a vertex $w$ such that $N(w) = N(v)$, then changing the outcome of the trial determining $\phi(w)$ could make $\aberrance_{v, \sigma}$ become 0 (if every $u \in N(v)$ is assigned a color corresponding to $\phi(w)$).  Therefore in order to apply Theorem~\ref{mr tala} to $\aberrance_{v, \sigma}$, the value of $c$ needs to be at least $d(v)$, and we do not get a useful bound.

However, it is very unlikely that every neighbor of $v$ is assigned a color corresponding to $\phi(w)$.  It is possible to derive a concentration inequality from Talagrand's Inequality with conditions similar to Theorem~\ref{mr tala} that apply for all but an unlikely set of exceptional outcomes.  One example is ``Talagrand's Inequality V'' in the book of Molloy and Reed~\cite[Chapter 20]{MR00}; another example was proved by Bruhn and Joos~\cite[Theorem 12]{BJ15}.  These results are also not enough for us to prove Lemma~\ref{concentration lemma}, as we discuss later.  Now we need some definitions in order to state our concentration inequality.

\begin{define}
  Let $((\Omega_i, \Sigma_i, \mathbb P_i))_{i=1}^n$ be probability spaces, let $(\Omega, \Sigma, \mathbb P)$ be their product space, let $\Omega^* \subseteq \Omega$ be a set of \textit{exceptional outcomes}, and let $X : \Omega \rightarrow \mathbb R_{\geq0}$ be a non-negative random variable.  Let $r, \change \geq 0$.
  \begin{itemize}
  \item If $\omega = (\omega_1, \dots, \omega_n) \in \Omega$ and $s > 0$, an \textit{$(r, \change)$-certificate} for $X, \omega, s$, and $\Omega^*$ is an index set $I\subseteq\{1, \dots, n\}$ of size at most $rs$ such that for all $k\geq 0$, we have that
  \begin{equation*}
    X(\omega') \geq s - k\change,
  \end{equation*}
  for all $\omega' = (\omega'_1, \dots, \omega'_n)\in\Omega\setminus\Omega^*$  such that $\omega_i\neq\omega_i'$ for at most $k$ values of $i\in I$.
  \item If for every $s > 0$ and $\omega\in\Omega\setminus\Omega^*$ such that $X(\omega) \geq s$, there exists an $(r, \change)$-certificate for $X, \omega, s$, and $\Omega^*$, then $X$ is \textit{$(r, \change)$-certifiable} with respect to $\Omega^*$.
  \end{itemize}
\end{define}

Note that if $\Omega^* = \varnothing$, then a random variable being $(r, \change)$-certifiable with respect to $\Omega^*$ is similar to it satisfying the conditions of Theorem~\ref{mr tala} with $c = \change$ (we use $\change$ because later we use $c$ to denote a color).  We introduce $k$ into the definition of $(r, \change)$-certificates rather than consider changing the outcome of only one trial because it is necessary in order to apply the original form of Talagrand's Inequality, for reasons we will see in Appendix~\ref{tala proof section}.  

Now we state our concentration inequality, as follows.

\begin{thm}\label{exceptional talagrand's}
  Let $((\Omega_i, \Sigma_i, \mathbb P_i))_{i=1}^n$ be probability spaces, let $(\Omega, \Sigma, \mathbb P)$ be their product space, let $\Omega^* \subseteq \Omega$ be a set of exceptional outcomes, and let $X : \Omega \rightarrow \mathbb R_{\geq0}$ be a non-negative random variable.  Let $r, \change \geq 0$.
  
  If $X$ is $(r, \change)$-certifiable with respect to $\Omega^*$, 
then for any $t > 96\change\sqrt{r\Expect{X}} +  128r\change^2 + 8\Prob{\Omega^*}(\sup X),$
\begin{equation*}
\Prob{|X - \Expect{X}| > t} \leq 4\exp\left({\frac{-t^2}{8\change^2r(4\Expect{X} + t)}}\right) + 4\Prob{\Omega^*}.
\end{equation*}
\end{thm}

Theorem~\ref{exceptional talagrand's} is similar to Theorem 12 of Bruhn and Joos~\cite{BJ15}.  Bruhn and Joos defined \textit{upward $(s, c)$-certificates}.  If a random variable is $(r, \change$)-certifiable with respect to a set of exceptional outcomes $\Omega^*$, then it has upward $(s, c)$-certificates with $c = \change$ and $s = r\cdot \sup X$, and for the random variables with which we are concerned, they have upward $(s, c)$-certificates only if $s \geq \sup X$.  The important difference between the bounds supplied by their result and Theorem~\ref{exceptional talagrand's} is that we have $r(4\Expect{X} + t)$ whereas they simply have $s$.  Bruhn and Joos~\cite{BJ15} apply their concentration inequality to random variables for which $\Expect{X} = \Omega(\sup X)$, so this difference does not concern them.  However, as mentioned, in our situation it is possible that $\sup X = \Delta$ and yet $\Expect{X} = \log^{10}\Delta$.  Thus, we are unable to use the result of Bruhn and Joos to prove Lemma~\ref{concentration lemma}.  ``Talagrand's Inequality V'' in \cite{MR00} has essentially the same problem, with $D$ taking the role of $s$.  We prove Theorem~\ref{exceptional talagrand's} in Appendix~\ref{tala proof section}; our proof is similar to the proof of Bruhn and Joos.

\begin{remark}\label{talagrand correction remark}
  We should expect Theorem~\ref{mr tala} to effectively follow from Theorem~\ref{exceptional talagrand's} in the case when $\Omega^* = \varnothing$, but this is not the case due to the presence of the $128r\change^2$ term in the lower bound on $t$ in the hypothesis.  Molloy and Reed's~\cite{MR14} corrected version of Theorem~\ref{mr tala} similarly introduces a $64rc^2$ term inside of the probability that $X$ deviates from its expectation.  Fortunately, in all applications of Theorem~\ref{mr tala} that we know of, we can still apply Theorem~\ref{exceptional talagrand's} with this additional term.  Indeed, in most applications $r$ and $\change$ are constants and $t$ is arbitrarily large, and almost always $r\change^2 = o\left(\Expect{X}\right)$, in which case this term is subsumed by the $96\change\sqrt{r\Expect{X}}$ term.  In this paper, $\change$ may be $\log^3\Delta$ and $\Expect{X}$ may be small, but we always apply Theorem~\ref{exceptional talagrand's} with $t \geq \log^9\Delta$.
\end{remark}

The exceptional outcomes we consider when applying Theorem~\ref{exceptional talagrand's} will involve many neighbors of a vertex $v$ receiving the same color (from some vertex $w$'s perspective), so we need this to be unlikely.
This explains why we need $\sigma < 1$ to apply Theorem~\ref{exceptional talagrand's} to $\aberrance_{v, \sigma}$, $\pairs_{v, \sigma}$, and $\trips_{v, \sigma}$.  In the extreme case, a vertex $v$ could have many neighbors with only two available colors, one of which does not correspond to a color in $L(v)$ and one of which corresponds to the same color for $v$.  Switching the color of a vertex may cause many neighbors of $v$ to become uncolored (or colored), which will significantly affect either $\pairs_{v, \sigma}$ and $\trips_{v, \sigma}$ or $\aberrance_{v, \sigma}$, or all three.  However, it is unlikely that many $\sigma$-egalitarian neighbors of $v$ receive the same color, as long as $|L(v)|$ is large.

We always apply Theorem \ref{exceptional talagrand's} with $t = \max\{\Expect{X}^{5/6}, \log^{\logexpless}\Delta\}$, $r\leq 9$, and $\change \leq \log^3 \Delta$.  Note that, assuming $\Delta$ is sufficiently large and $\Prob{\Omega^*}$ is sufficiently small, $t$ is large enough to apply Theorem~\ref{exceptional talagrand's}.

The following proposition will be useful.
\begin{prop}\label{simplify t}
  If $X$ is a non-negative random variable and $t = \max\{\gamma\cdot\Expect{X}^{5/6}, \log^{\logexpless}\Delta\}$ where $\gamma > 0$, then
  \begin{equation*}
    \frac{t^2}{4\Expect{X} + t} \geq \frac{\log^{36/5}\Delta}{1 + 4/\gamma^{6/5}}.
  \end{equation*}
\end{prop}
\begin{proof}
  Since $\Expect{X} \leq (t/\gamma)^{6/5}$,
  \begin{equation*}
    \frac{t^2}{4\Expect{X} + t} \geq \frac{t^2}{4(t/\gamma)^{6/5} + t} \geq \frac{t^{4/5}}{1 + 4/\gamma^{6/5}}.
  \end{equation*}
  Since $t \geq \log^{\logexpless}\Delta$, the result follows.
\end{proof}

The following proposition bounds the probability that many non-subservient neighbors of a vertex receive the same color.
\begin{prop}\label{many colors exceptional outcome}
  For each $v\in V(G)$, let $\Omega^*_{v,\sigma}$ be the set of events where there exists $u\in V(G), c\in L(u)$, and a set $X\subset (\egal_\sigma(v)\cap N(u))$ of size at least $\log\Delta$ such that for each $w\in X$, we have that $\phi(w)c\in M_{wu}$.  Now
  \begin{equation*}
    \Prob{\Omega^*_{v, \sigma}} \leq \Delta^4\left(\frac{e}{(1 - \sigma)(1 - \varepsilon)\log\Delta}\right)^{\log\Delta}.
  \end{equation*}
\end{prop}
\begin{proof}
  For each $u\in V(G)$ and $c\in L(u)$, let
  \begin{equation*}
   Y_{u,c} = |\{w\in (N(u)\cap \egal_\sigma(v)) : \phi(w)c\in M_{wu}\}|.
  \end{equation*}
  Now
  \begin{equation*}
    \Prob{Y_{u,c} \geq \log\Delta} \leq \sum_{i=\lceil \log \Delta\rceil}^{d(v)}{d(v)\choose i}\frac{1}{((1 - \sigma)|L(v)|)^i}.
  \end{equation*}

  By applying the bound $\binom{d(v)}{i} < \left(\frac{e\cdot d(v)}{i}\right)^i$ and using the fact that $\frac{1}{|L(v)|} \leq \frac{1}{(1 - \varepsilon)d(v)}$,
  \begin{equation*}
    \Prob{Y_{u,c} \geq \log\Delta} \leq \sum_{i=\lceil \log \Delta\rceil}^{d(v)}\left(\frac{e\cdot d(v)}{i}\right)^i\frac{1}{(1 - \varepsilon)^id(v)^i} = \sum_{i=\lceil \log \Delta\rceil}^{d(v)}\left(\frac{e}{(1 - \sigma)(1 - \varepsilon) i}\right)^i.
  \end{equation*}

  Since each term in the sum is at most $\left(\frac{e}{(1 - \sigma)(1 - \varepsilon)\log\Delta}\right)^{\log\Delta}$ and there are at most $\Delta$ terms, it follows that  
  \begin{equation*}
    \Prob{Y_{u,c}\geq\log\Delta} \leq \Delta\left(\frac{e}{(1 - \sigma)(1 - \varepsilon)\log\Delta}\right)^{\log \Delta}.
  \end{equation*}

  Since $|N(u)\cap \egal_\sigma(v)| = 0$ for all but $\Delta^2$ vertices $u$, and each has at most $\Delta$ available colors, by the Union Bound,
  \begin{equation*}\label{lordly exceptional probability}
    \Prob{\Omega^*_{v, \sigma}} \leq \Delta^4\left(\frac{e}{(1 - \sigma)(1 - \varepsilon)\log\Delta}\right)^{\log\Delta},
  \end{equation*}
  as desired.
\end{proof}
Observe that $\Prob{\Omega^*_{v, \sigma}} = o(\Delta^{-4})$.

Now we can prove Lemma~\ref{concentration lemma}.  For each $v\in V(G)$, let $(\Omega_{v, 1}, \Sigma_{v, 1}, \mathbb P_{v, 1})$ be the probability space where $\Omega_{v, 1} = L(v)$, the sigma-algebra $\Sigma_{v, 1}$ is the discrete sigma-algebra, and $\mathbb P_{v, 1}$ is the uniform distribution (i.e.\ this probability space corresponds to assigning $v$ a color from $L(v)$ uniformly at random), let $(\Omega_{v, 2}, \Sigma_{v, 2}, \mathbb P_{v, 2})$ be the probability space where $\Omega_{v, 2} = \{\text{heads}, \text{tails}\}$, the sigma-algebra $\Sigma_{v, 2}$ is again discrete, and $\mathbb P_{v, 2}[\text{heads}] = \rho$ (i.e.\ this probability space corresponds to activating $v$ with probability $\rho$), and for each $c \in L(v)$, let $(\Omega_{v, c}, \Sigma_{v, c}, \mathbb P_{v, c})$ be the probability space where $\Omega_{v, c} = \{\text{heads}, \text{tails}\}$, the sigma-algebra $\Sigma_{v, c}$ is again discrete, and $\mathbb P_{v, c}[\text{heads}] = 1 - K_{\varepsilon, \rho} / p$, where $p$ is the probability that $v$ is not uncolored after an application of the local naive random coloring procedure with activation probability $\rho$, conditioned on the event that $v$ is assigned color $c$ (i.e.\ this probability space corresponds to an $\varepsilon$-equalizing coin-flip for $v$ and $c$).  Let $(\Omega, \Sigma, \mathbb P)$ be the product space of $(\Omega_{v, i}, \Sigma_{v, i}, \mathbb P_{v, i})_{v\in V(G), i\in\{1,2\}}$ and $(\Omega_{v, c}, \Sigma_{v, c}, \mathbb P_{v, c})_{v\in V(G), c\in L(v)}$.  In order to sample a naive partial coloring using the local naive random coloring procedure with activation probability $\rho$ and $\varepsilon$-equalizing coin flips, we sample from $\Omega$.  If $\omega$ is an outcome in $\Omega$, then we let $(\phi_\omega, \uncolvtcs_\omega)$ be the corresponding naive partial coloring and let $A_\omega$ be the set of activated vertices. 
We prove each random variable is $\Delta$-concentrated individually, as follows.

\begin{proof}[Proof that $\unact_{v, \ordering}$ is $\Delta$-concentrated]
  We claim that $\unact_{v, \ordering}$ is $(r, \change)$-certifiable with respect to $\Omega^* = \varnothing$, where $r = 1$ and $\change = 1$.  Let $s > 0$ and let $\omega\in\Omega$ such that $\unact_{v, \ordering}(\omega) \geq s$.  We show that there is an $(r, \change)$-certificate, $I$, for $\unact_{v, \ordering}, \omega, s,$ and $\varnothing$.

  Since $\unact_{v, \ordering}(\omega) \geq s$, there is a set $S$ of $s$ neighbors $u$ of $v$ such that $u\ordering v$ and $u\notin A_\omega$.  
  Thus,
  \begin{equation}
    \label{subservience lower bound 1}
    \unact_{v, \ordering}(\omega) \geq |S|  = s.
  \end{equation}
  We let $I$ index the trials determining if $u \in A_\omega$ for the $u\in S$, so $|I| = s$. 
  
  We claim that $I$ is an $(r, \change)$-certificate for $\unact_{v, \ordering}, \omega, s,$ and $\varnothing$.  To that end, let $\omega'\in \Omega$ and $k\geq 0$ such that $\omega$ and $\omega'$ differ for at most $k$ trials indexed by $I$.  
  Let $T = S \setminus A_{\omega'}$. 
  Note that
  \begin{equation}\label{subservience lower bound 2}
    \unact_{v, \ordering}(\omega') \geq |T|.
  \end{equation}
  Since $\omega$ and $\omega'$ differ in at most $k$ trials indexed by $I$,
  \begin{equation}\label{subservience different outcomes}
    |S\setminus T| \leq k. 
  \end{equation}
  Therefore by \eqref{subservience lower bound 1}, \eqref{subservience lower bound 2}, and \eqref{subservience different outcomes}, $\unact_{v, \ordering}(\omega') \geq s - k\change$, so $I$ is an $(r, \change)$-certificate for $\unact_{v, \ordering}, \omega, s,$ and $\varnothing$, as claimed.  Thus, $\unact_{v, \ordering}$ is $(r, \change)$-certifiable with respect to $\varnothing$, as claimed, and we can apply Theorem~\ref{exceptional talagrand's}.
  We choose $t = \max\{\Expect{\unact_{v, \ordering}}^{5/6}, \log^{\logexpless}\Delta\}$, so by Proposition~\ref{simplify t} and Theorem~\ref{exceptional talagrand's}, for some constant $\gamma_1 > 0$,
  \begin{equation*}
    \Prob{|\unact_{v, \ordering} - \Expect{\unact_{v, \ordering}}| > t} \leq 4\exp(-\gamma_1(\log^{26/5}(\Delta))).
  \end{equation*}
  Since $\Delta$ is sufficiently large, the result follows.
\end{proof}

\begin{proof}[Proof that $\aberrance_{v,\sigma}$ is $\Delta$-concentrated]
  We cannot show that $\aberrance_{v,\sigma}$ is $(r, \change)$-certifiable with respect to any appropriate set of exceptional outcomes, but we can express $\aberrance_{v,\sigma}$ as the difference of two random variables that are.  To that end, we define the following random variables in which $(\phi, \uncolvtcs)$ is a random naive partial coloring:
  \begin{align*}
    &\totesabs_v = |\{u \in \egal_\sigma(v) : \phi(u)\notin V(M_{uv})\}|,\text{ and}\\
    &\uncolabs_v = |\{u \in \egal_\sigma(v)\cap \uncolvtcs : \phi(u)\notin V(M_{uv})\}|.
  \end{align*}
  
  Note that $\totesabs_v$ is $(r, \change)$-certifiable with respect to $\Omega^* = \varnothing$, where $r, \change =1$, by the same argument as in the proof that $\unact_{v, \ordering}$ is $\Delta$-concentrated.
  Note also that $\Expect{\aberrance_{v,\sigma}} = K_{\varepsilon, \rho}\cdot\Expect{\totesabs_v} = K_{\varepsilon, \rho}\cdot\Expect{\uncolabs_v}/(1 - K_{\varepsilon, \rho})$.
  Let $t = \max\{\Expect{\aberrance_{v,\sigma}}^{5/6},\allowbreak \log^{\logexpless}\Delta\}$, so by Proposition~\ref{simplify t} and Theorem~\ref{exceptional talagrand's}, for some constant $\gamma_2 > 0$,
  \begin{equation}\label{total aberrance concentration}
    \Prob{|\totesabs_v - \Expect{\totesabs_v}| > t} \leq 4\exp(-\gamma_2(\log^{36/5}(\Delta))).
  \end{equation}
  
  Now we show that $\uncolabs_v$ is $(r, \change)$-certifiable with exceptional outcomes $\Omega^*_{v, \sigma}$ from Proposition \ref{many colors exceptional outcome} with $r = 3$ and $\change = \log\Delta$.  Let $s > 0$ and let $\omega \in \Omega \setminus\Omega^*_{v, \sigma}$ such that $\uncolabs_v(\omega) \geq s$.
    Since $\uncolabs_v(\omega) \geq s$, there is a set $S_1$ of $s$ $\sigma$-egalitarian neighbors $u$ of $v$ such that $\phi(u) \notin V(M_{uv})$ and $u\in U$.  Each such vertex $u\in S_1$ either has a neighbor $u'\in A_\omega$ such that $|L(u')| \geq |L(u)|$ and $\phi_\omega(u)\phi_\omega(u')\in M_{uu'}$, is itself not in $A_\omega$, or is uncolored by an $\varepsilon$-equalizing coin-flip.  In the first case, we choose precisely one such neighbor $u'$ of $u$, let $u'$ be in the set $S_2$, and let $u$ be in the set $S_{u'}$.  In the second case, let $u \in S'_1$, and in the third case, we let $u\in S''_1$.  By the definition of these sets,
  \begin{equation}
    \label{unmatched lower bound 1}
    \uncolabs_v(\omega) \geq |S'_1| + |S''_1| + \sum_{u \in S_2}|S_u|  = s.
  \end{equation}
  We let $I$ index the trials determining if $u \in A_\omega$ for the $u\in S'_1 \cup S_2$, we let $I$ index the trial for the $\varepsilon$-equalizing coin flip for $u$ and $\phi_\omega(u)$ for each $u \in S''_1$, and for each $u\in S_1\setminus (S'_1 \cup S''_1)$, there exists $u'\in N(u)\cap S_2$, and we let $I$ index the trial determining $\phi_\omega(u')$.  We also let $I$ index the trial determining $\phi_\omega(u)$ for each $u \in S_1$.  Note that $|I| \leq 3s$.  
  
  We claim that $I$ is an $(r, \change)$-certificate for $\uncolabs_v, \omega, s,$ and $\Omega^*_{v, \sigma}$.  To that end, let $\omega'\in \Omega\setminus\Omega^*_{v, \sigma}$ and $k\geq 0$ such that $\omega$ and $\omega'$ differ for at most $k$ trials indexed by $I$.  We say a vertex \textit{keeps its color} if $\phi_\omega(u) = \phi_{\omega'}(u)$.
  Let $T'_1$ be the set of vertices in $S'_1 \setminus A_{\omega'}$ that keep their color, let $T''_1$ be the set of vertices in $S''_1$ that keep their color and are also uncolored by an $\varepsilon$-equalizing coin-flip in the outcome $\omega'$, let $T_2$ be the set of vertices in $S_2 \cap A_{\omega'}$ that keep their color, and for each $u\in T_2$, let $T_u$ be the set of vertices in $S_u$ that keep their color.
  Note that
  \begin{equation}\label{unmatched lower bound 2}
    \uncolabs_v(\omega') \geq |T'_1| + |T''_1| + \sum_{u\in T_2}|T_u|.
  \end{equation}
  Moreover, the sets in the above inequality are pairwise disjoint.
  Since $\omega$ and $\omega'$ differ in at most $k$ trials indexed by $I$,
  \begin{equation}\label{unmatched different outcomes}
    |S'_1\setminus T'_1| + |S''_1\setminus T''_1| + |S_2\setminus T_2| + |\cup_{u\in T_2}S_u\setminus T_u| \leq k.
  \end{equation}
  Since $\omega\notin\Omega^*_{v, \sigma}$, for each $u\in S_2$, we have that $|S_u| \leq \log \Delta = \change$.
  Therefore by \eqref{unmatched lower bound 1}, \eqref{unmatched lower bound 2}, and \eqref{unmatched different outcomes}, $\uncolabs_v(\omega') \geq s - k\change$, so $I$ is an $(r, \change)$-certificate for $\uncolabs_v, \omega, s,$ and $\Omega^*_{v, \sigma}$, as claimed.

  By Proposition~\ref{simplify t} and Theorem~\ref{exceptional talagrand's}, for some constant $\gamma_3 > 0$,
\begin{equation}\label{uncolored aberrance concentration}
  \Prob{|\uncolabs_v - \Expect{\uncolabs_v}| > t} \leq 4\exp(-\gamma_3(\log^{26/5}(\Delta))) + 4\Prob{\Omega^*_v}.
\end{equation}

Since $\aberrance_{v,\sigma} = \totesabs_v - \uncolabs_v$, it follows from \eqref{total aberrance concentration}, \eqref{uncolored aberrance concentration}, and Proposition \ref{many colors exceptional outcome} that $\aberrance_{v,\sigma}$ is $\Delta$-concentrated, as desired.
\end{proof}

\begin{proof}[Proof that $\pairs_{v, \sigma}$ and $\trips_{v, \sigma}$ are $\Delta$-concentrated]
  As in the proof that $\aberrance_{v,\sigma}$ is $\Delta$-concentrated, we do not show that $\pairs_{v, \sigma}$ and $\trips_{v, \sigma}$ are $(r, d)$-certifiable with respect to some set of exceptional outcomes.  Instead, we express $\pairs_{v, \sigma}$ and $\trips_{v, \sigma}$ as differences of such random variables and apply Theorem~\ref{exceptional talagrand's} to each of these new random variables.    If $H$ is a graph, recall that $T(H)$ denotes the set of triangles in $H$.  We define the following random variables in which $(\phi, \uncolvtcs)$ is a random naive partial coloring:
  \begin{align*}
    &\totespairs_{v, \sigma} = |\{x,y\in \egal_\sigma(v), c\in L(v) : \phi(x)c\in M_{xv}\text{ and }\phi(y)c\in M_{yv}\}|,\\
    &\totestrips_{v, \sigma} = |\{x,y,z \in \egal_\sigma(v), c\in L(v) : \phi(x)c\in M_{xv}, \phi(y)c\in M_{yv},\text{ and }\phi(z)c\in M_{zv}\}|,\\
    \begin{split}
      &\uncolpairs_{v, \sigma} = |\{x,y\in \egal_\sigma(v), c\in L(v) : \{x,y\}\cap \uncolvtcs\neq\varnothing,\\
      &\qquad \phi(x)c\in M_{xv}\text{ and }\phi(y)c\in M_{yv}\}|,\text{ and}
    \end{split}\\
    \begin{split}
      &\uncoltrips_{v, \sigma} = |\{x,y,z \in \egal_\sigma(v), c\in L(v), c\in L(v) : \{x,y,z\}\cap \uncolvtcs\neq\varnothing, \\
      &\qquad \phi(x)c\in M_{xv}, \phi(y)c\in M_{yv},\text{ and }\phi(z)c\in M_{zv}\}|.
    \end{split}
  \end{align*}

  Note that $\pairs_{v, \sigma} = \totespairs_{v, \sigma} - \uncolpairs_{v, \sigma}$ and $\trips_{v, \sigma} = \totestrips_{v, \sigma} - \uncoltrips_{v, \sigma}$.
  Note also that $\Expect{\pairs_{v, \sigma}} = \Theta(\Expect{\totespairs_{v, \sigma}}) = \Theta(\Expect{\uncolpairs_{v, \sigma}})$ and $\Expect{\trips_{v, \sigma}} = \Theta(\Expect{\totestrips_{v, \sigma}}) = \Theta(\Expect{\uncoltrips_{v, \sigma}}).$

  We claim that $\totespairs_{v, \sigma}$ and $\uncolpairs_{v, \sigma}$ are $(r, \change)$-certifiable with respect to exceptional outcomes $\Omega^*_{v, \sigma}$ from Proposition~\ref{many colors exceptional outcome}, where $r = 6$ and $\change = \log^2 \Delta$.  We only provide a proof for $\uncolpairs_{v, \sigma}$, since the proof for $\totespairs_{v, \sigma}$ is easier.  Let $s > 0$ and let $\omega\notin\Omega^*_{v, \sigma}$ such that $\uncolpairs_{v, \sigma}(\omega) \geq s$.  We show that there is an $(r, \change)$-certificate, $I$, for $\uncolpairs_{v, \sigma}, \omega, s,$ and $\Omega^*_{v, \sigma}$.

  For each $c\in L(v)$, define $S_{c, 1}$ as follows.  If the set of uncolored $\sigma$-egalitarian neighbors $u$ of $v$ such that $\phi_\omega(u)c \in M_{uv}$ has size at least two, then let that set be $S_{c, 1}$.  Otherwise, let $S_{c, 1} = \varnothing$.  For each $c\in L(v)$, each vertex $u\in S_{c, 1}$ either has a neighbor $u' \in A_\omega$ such that $|L(u')| \geq |L(u)|$ and $\phi_\omega(u)\phi_\omega(u') \in M_{uu'}$, is itself not in $A_\omega$, or is uncolored by an $\varepsilon$-equalizing coin-flip.  In the first case, we choose precisely one such neighbor $u'$ of $u$, let $u'$ be in the set $S_{c, 2}$, and let $u$ be in the set $S_{c, u'}$.  In the second case, we let $u \in S'_{c, 1}$, and in the third case, we let $u\in S''_{c, 1}$.  By the definition of these sets,
  \begin{equation}\label{uncolpairs for omega}
    \uncolpairs_{v, \sigma}(\omega) = \sum_{c\in L(v)}\binom{|S_{c, 1}|}{2}
  \end{equation}
  and
  \begin{equation*}
    S_{c, 1} = S'_{c, 1}\cup S''_{c, 1} \cup\left(\bigcup_{u\in S_{c, 2}}S_{c, u}\right).
  \end{equation*}
  Since $\omega\notin\Omega^*_{v, \sigma}$, for each $c\in L(v)$, we have that $|S_{c, 1}| \leq \log\Delta$, and for each $u\in S_{c, 2}$, we have that $|\cup_{c'\in L(v)}S_{c', u}| \leq \log\Delta$.

  For each $c\in L(v)$, we let $I_c$ index the trials determining $\phi_\omega(u)$ for the $u\in S_{c, 1}\cup S_{c, 2}$, for each $u\in S'_{c, 1} \cup S_2$, we let $I_c$ index the trial determining if $u \in A_\omega$, and for each $u \in S''_{c, 1}$, the vertex $u$ is uncolored by an $\varepsilon$-equalizing coin flip, and we also let $I_c$ index this trial.  We let $I = \cup_{c\in L(v)}I_c$.  

  We claim that $I$ is an $(r, \change)$-certificate for $\uncolpairs_{v, \sigma}, \omega, s$, and $\Omega^*_{v, \sigma}$.  To that end, let $\omega' \in \Omega\setminus\Omega^*_{v, \sigma}$ and $k\geq 0$ such that $\omega$ and $\omega'$ differ for at most $k$ trials indexed by $I$.  We say a vertex $u$ \textit{keeps its color} if $\phi_\omega(u) = \phi_{\omega'}(u)$.  For each $c\in L(v)$, let $T'_{c, 1}$ be the set of vertices in $S'_{c, 1}\setminus A_{\omega'}$ that keep their color, let $T''_{c, 1}$ be the set of vertices in $S''_{c, 1}$ that keep their color and are also uncolored by an $\varepsilon$-equalizing coin-flip in the outcome $\omega'$, let $T_{c, 2}$ be the set of vertices in $S_{c, 2} \cap A_{\omega'}$ that keep their color, and for each $u\in T_{c, 2}$ , let $T_{c, u}$ be the set of vertices in $S_{c, u}$ that keep their color.  For each $c\in L(v)$, let $T_{c, 1} = T'_{c, 1}\cup T''_{c, 1}\cup\left(\bigcup_{u\in T_{c, 2}}T_{c, u}\right)$.  Note that
  \begin{equation}\label{uncolpairs for omega'}
    \uncolpairs_{v, \sigma}(\omega') \geq \sum_{c\in L(v)}\binom{\left|T_{c, 1}\right|}{2}.
  \end{equation}
  and
  \begin{equation}\label{uncolpairs difference}
    \sum_{c\in L(v)}\left(\binom{|S_{c, 1}|}{2} - \binom{|T_{c, 1}|}{2}\right) = \sum_{c\in L(v)}|S_{c, 1}\setminus T_{c, 1}|(|S_{c, 1}| + |T_{c, 1}| - 1)/2.
  \end{equation}
  Recall that for each $c\in L(v)$ and $u\in S_{c, 2}$, we have that $|\cup_{c'\in L(v)}S_{c', u}| \leq \log \Delta$.  Since $\omega$ and $\omega'$ differ for at most $k$ trials indexed by $I$, it follows that $\sum_{c\in L(v)}|S_{c, 1}\setminus T_{c, 1}| \leq k\log\Delta$.  Also note that $|S_{c, 1}| + |T_{c, 1}| - 1)/2 \leq \log \Delta$.  Therefore by \eqref{uncolpairs for omega}, \eqref{uncolpairs for omega'}, and \eqref{uncolpairs difference}, $\uncolpairs_{v, \sigma}(\omega') \geq s - k\change$, as required.

  Note that for each $c\in L(v)$, we have that $|I_c| \leq 3|S_{c, 1}|$, and hence $|I_c| \leq 6\binom{|S_{c, 1}|}{2}$.  Therefore $|I| \leq 6s$, as required.  It follows that $I$ is an $(r, \change)$-certificate for $\uncolpairs_{v, \sigma}, \omega, s$, and $\Omega^*_{v, \sigma}$, and hence $\uncolpairs_{v, \sigma}$ is $(r, \change)$-certifiable with respect to $\Omega^*_{v, \sigma}$, as claimed.  Therefore we can apply Theorem~\ref{exceptional talagrand's}.

  We choose $t = \max\{\Expect{\pairs_{v, \sigma}}^{5/6}, \log^{\logexpless}\Delta\}$, so by Proposition~\ref{simplify t} and Theorem~\ref{exceptional talagrand's}, for some constant $\gamma_4 > 0$,
  \begin{equation}
    \label{total pairs concentration}
    \Prob{|\totespairs_{v, \sigma} - \Expect{\totespairs_{v, \sigma}}| > t} \leq 4\exp(-\gamma_4(\log^{16/5}(\Delta))) + 4\Prob{\Omega^*_v},
  \end{equation}
  and
  \begin{equation}
    \label{uncolored pairs concentration}
    \Prob{|\uncolpairs_{v, \sigma} - \Expect{\uncolpairs_{v, \sigma}}| > t} \leq 4\exp(-\gamma_4(\log^{16/5}(\Delta))) + 4\Prob{\Omega^*_v}.
  \end{equation}
  It follows from \eqref{total pairs concentration}, \eqref{uncolored pairs concentration}, and Proposition~\ref{many colors exceptional outcome} that $\pairs_{v, \sigma}$ is $\Delta$-concentrated, as desired.

  Similarly, we can apply Theorem~\ref{exceptional talagrand's} to $\totestrips_{v, \sigma}$ and $\uncoltrips_{v, \sigma}$ with exceptional outcomes $\Omega^*_{v, \sigma}$, $r = 9$, and $\change = \log^3\Delta$.  Letting $t = \max\{\Expect{\trips_{v, \sigma}}^{5/6}, \log^{\logexpless}\Delta\}$, we observe that for some constant $\gamma_5 > 0$,
  \begin{equation}
    \label{total trips concentration}
    \Prob{|\totestrips_{v, \sigma} - \Expect{\totestrips_{v, \sigma}}| > t} \leq 4\exp(-\gamma_5(\log^{6/5}(\Delta))) + 4\Prob{\Omega^*_v},
  \end{equation}
  and
  \begin{equation}
    \label{uncolored trips concentration}
    \Prob{|\uncoltrips_{v, \sigma} - \Expect{\uncoltrips_{v, \sigma}}| > t} \leq 4\exp(-\gamma_5(\log^{6/5}(\Delta))) + 4\Prob{\Omega^*_v}.
  \end{equation}
  It follows from \eqref{total trips concentration}, \eqref{uncolored trips concentration}, and Proposition~\ref{many colors exceptional outcome} that $\trips_{v, \sigma}$ is $\Delta$-concentrated, as desired.
\end{proof}

\section{Critical Graphs and Maximum Average Degree}\label{critical section}

In this section we prove Theorems~\ref{critical thm} and \ref{epsilon mad}.

\subsection{Proof of Theorem~\ref{critical thm}}

We prove Theorem~\ref{critical thm} by finding an appropriate induced subgraph $G'$ of the graph $G$, using the criticality of $G$ to $L$-color $G - V(G')$, and then using Theorem~\ref{main thm} to extend this coloring to an $L$-coloring of $G$, contradicting the criticality of $G$.  In order to extend the $L$-coloring of $G - V(G')$ to one of $G$ using Theorem~\ref{main thm}, the vertices of $G'$ need to have few neighbors in $G - V(G')$.  The following lemma provides the existence of such a subgraph.
\begin{lemma}\label{critical subgraph lemma}
For every $1 \geq \alpha > \varepsilon > 0$, every graph $H$ with $\ad(H) \leq (1 + \varepsilon)\delta(H)$ contains a nonempty induced subgraph $H'\subseteq H$ such that for every $v\in V(H')$
\begin{enumerate}
	\item $d_{H'}(v) \geq \left(\frac{1 - \alpha}{2}\right)\delta(H)$ and
	\item $d_H(v) \leq \left(1 + \frac{1 + \alpha}{\alpha - \varepsilon}\varepsilon\right)\delta(H)$.
\end{enumerate}
\end{lemma}
\begin{proof}
  We use the discharging method.  For each $v\in V(G)$, let the \textit{charge} of $v$ be $ch(v) = d(v) - \ad(H)$.  Note that $\sum_{v\in V(G)}ch(v) = 0$.  Let $X$ denote the set of vertices of $H$ with degree greater than $\left(1 + \frac{1 + \alpha}{\alpha - \varepsilon}\varepsilon\right)\delta(H)$.  Note that $X$ is a proper subset of the vertices of $H$ since $\ad(H) \leq (1 + \varepsilon)\delta(H)$.  We may assume $\delta(H - X) < \left(\frac{1 - \alpha}{2}\right)\delta(H)$ or else $H - X$ is the desired induced subgraph.

  We redistribute the charges in the following way.  Let every $v\in X$ send $ch(v)/d(v)$ charge to each of its neighbors.  Note that for every $v\in X$,
  \begin{equation*}
    \frac{ch(v)}{d(v)} = 1 - \frac{\ad(H)}{d(v)} > 1 - \frac{\ad(H)}{\left(1 + \frac{1 + \alpha}{\alpha - \varepsilon}\varepsilon\right)\delta(H)} \geq \frac{\varepsilon}{\alpha}.
  \end{equation*}
  Therefore every vertex in $X$ has zero charge, and every $v\in V(H - X)$ has charge at least $d_H(v) - \ad(H) + \frac{\varepsilon}{\alpha}(d_H(v) - d_{H - X}(v))$.  If $d_H(v) - d_{H - X}(v) > 0$, then the inequality is strict.
  
  Now we claim we can iteratively remove vertices from $H - X$ of minimum degree to obtain a nonempty graph of minimum degree at least $\left(\frac{1 - \alpha}{2}\right)\delta(H)$.  When we remove a vertex of $H - X$, we add it to a new set $X'$, and we let it send charge $\frac{\varepsilon}{\alpha}$ to every neighbor not in $X\cup X'$.  It suffices to show that every vertex in $X'$ has nonnegative charge and that at least one vertex in $X'$ has positive charge, because then the sum of the charges taken over vertices in $H - (X \cup X')$ is negative, and thus $H' = H - (X \cup X')$ is nonempty.

  Note that if $v\notin X\cup X'$ has degree at most $\left(\frac{1 - \alpha}{2}\right)\delta(H)$ in $H - (X\cup X')$, then $v$ has at least $\left(\frac{1 + \alpha}{2}\right)\delta(H)$ neighbors in $X\cup X'$.  Therefore $v$ receives at least $\frac{\varepsilon}{\alpha}(\frac{1 + \alpha}{2})\delta(H)$ charge and sends at most $\frac{\varepsilon}{\alpha}\left(\frac{1 - \alpha}{2}\right)\delta(H)$ charge.  Hence the difference in charge received and sent is at least $\varepsilon\delta(H)$, and if $v$ has a neighbor in $X$, the inequality is strict.  Therefore $v$ has nonnegative charge, and since at least one vertex of $X'$ has a neighbor in $X$, there is a vertex of $X'$ with positive charge, as desired.
\end{proof}

Now we can prove Theorem~\ref{critical thm}.
\begin{proof}[Proof of \ref{critical thm}]
  Let $\alpha > 0$, and let $\varepsilon \leq \frac{\alpha^2}{1350}$.  Let $G$ be an $L$-critical graph for some $k$-list-assignment $L$ such that $\omega(G)\leq (\frac{1}{2} - \alpha)k$.  Note then that $\alpha < \frac{1}{2}$.  Suppose for a contradiction that $\ad(G) \leq (1 + \varepsilon)k$.  Since $G$ is $L$-critical, $G$ has minimum degree at least $k$.  By Lemma~\ref{critical subgraph lemma}, there exists $G'\subseteq G$ such that for every $v\in V(G')$,
  \begin{enumerate}
  \item $d_{G'}(v) \geq \left(\frac{1 - \alpha}{2}\right)\delta(G)$, and
  \item $d_G(v) \leq \left(1 + \frac{1 + \alpha}{\alpha - \varepsilon}\varepsilon\right)\delta(G)$.
  \end{enumerate}
  Since $G$ is $L$-critical, $G - V(G')$ is $L$-colorable.  Let $\phi$ be an $L$-coloring of $G - V(G')$, and for each $v\in V(G')$, let
  \begin{equation*}
    L'(v) = L(v) \setminus \{c\in L(v) : \exists u\in N(v)\setminus V(G') : \phi(u) = c\}.
  \end{equation*}
  Note that $G'$ is not $L'$-colorable, because we can combine an $L'$-coloring of $G'$ with $\phi$ to obtain an $L$-coloring of $G$.

  Since $d_{G'}(v) \geq \left(\frac{1 - \alpha}{2}\right)\delta(G)$, $\delta(G) \geq k$, and $\omega(v) \leq \omega(G) \leq (\frac{1}{2} - \alpha)k$ for each $v\in V(G')$,
  \begin{equation*}
    \Gap_{G'}(v) \geq \frac{\alpha}{2}k.
  \end{equation*}
  Since each $v\in V(G')$ has at most $d_G(v) - d_{G'}(v)$ neighbors in $V(G)\setminus V(G')$,
  \begin{equation*}
    \Save_{L'}(v) \leq d_G(v) - k \leq \left(\left(1 + \frac{1 + \alpha}{\alpha - \varepsilon}\varepsilon\right)\left(1 + \varepsilon\right) - 1\right)k.
  \end{equation*}
  Since $\varepsilon \leq \frac{\alpha^2}{1350}$ and $\alpha < \frac{1}{2}$,
  \begin{equation*}
    \frac{1 + \alpha}{\alpha - \varepsilon}\varepsilon(1 + \varepsilon) + \varepsilon \leq\frac{\alpha}{1350}\left(\frac{(1 + \alpha)(1 + \alpha^2/1350)}{1 - \alpha/1350} + \alpha\right)\leq \frac{\alpha}{660}.
  \end{equation*}
  Therefore $\Save_{L'}(v) \leq \frac{\alpha}{660}k$.  Now for every vertex $v\in V(G')$, $\Save_{L'}(v) \leq \frac{1}{330}\Gap_{G'}(v)$ and for sufficiently large $k$, $\Gap_{G'}(v) - \Save_{L'}(v) \geq \log^{\logexp}(\Delta(G'))$.  Thus, by Theorem~\ref{main thm}, $G'$ is $L'$-colorable, a contradiction.
\end{proof}

\subsection{Proof of Theorem~\ref{epsilon mad}}

In this subsection we prove Theorem~\ref{epsilon mad}.  It follows fairly easily from Theorem~\ref{critical thm}.
\begin{proof}[Proof of Theorem~\ref{epsilon mad}]
  Given $\alpha > 0$, we let $\varepsilon > 0$ be some constant chosen to be small enough to satisfy certain inequalities throughout the proof.  Let $G$ be a graph such that $\omega(G) \leq (\frac{1}{2} - \alpha)\mad(G)$, and let
  \begin{equation*}
    k  = \left\lceil(1 - \varepsilon)(\mad(G) + 1) + \varepsilon\omega(G)\right\rceil.
  \end{equation*}

  First we prove that there exists an integer $k_0$ such that if $\mad(G) \geq k_0$, then $\chi_\ell(G) \leq k$.  We choose $k_0$ such that $k$ is large enough to apply Theorem~\ref{critical thm}.  Since $k \geq (1 - \varepsilon)\mad(G)$ and $(\frac{1}{2} - \alpha)\mad(G) \geq \omega(G)$, assuming $\varepsilon$ is small enough, $\omega(G) \leq (\frac{1}{2} - \frac{\alpha}{2})k$.

  Let $\varepsilon' > 0$ according to Theorem~\ref{critical thm} for $\alpha/2$.  We may assume $L$ is a $k$-list-assignment for $G$ such that $G$ is not $L$-colorable, or else $\chi_\ell(G) \leq k$, as desired.  Therefore $G$ contains an $L$-critical subgraph $G'$, and by Theorem~\ref{critical thm}, $\ad(G') \geq (1 + \varepsilon')k$.  Hence,
  \begin{equation*}
    (1 + \varepsilon')k \leq \mad(G) \leq \frac{k}{1 - \varepsilon}.
  \end{equation*}
  But we may assume $\varepsilon$ is sufficiently small so that $(1 + \varepsilon') > \frac{1}{1 - \varepsilon}$, a contradiction.  Therefore $\chi_\ell(G) \leq k$ if $\mad(G) \geq k_0$.

  It remains to show that $\chi_\ell(G) \leq k$ if $\mad(G) < k_0$.  If we choose $\varepsilon$ to be less than $\frac{1}{k_0 + 2}$, then
  \begin{equation*}
    k \geq \left\lceil\left(1 - \frac{1}{\mad(G) + 2}\right)\left(\mad(G) + 1\right)\right\rceil = \left\lceil \mad(G) + \frac{1}{\mad(G) + 2}\right\rceil \geq \lfloor \mad(G)\rfloor + 1.
  \end{equation*}
  Therefore we can obtain an $L$-coloring of $G$ for any $k$-list-assignment $L$ by coloring greedily.  Thus, $\chi_\ell(G) \leq k$, as desired. 
\end{proof}
 
\bibliographystyle{plain}
\bibliography{local-reed}

\appendix
\section{Proof of Theorem \ref{exceptional talagrand's}}\label{tala proof section}

In order to prove Theorem \ref{exceptional talagrand's}, we prove the following theorem which yields concentration around the median under the same conditions.
\begin{thm}\label{exceptional talagrand's with median}
  If $X$ is $(r, \change)$-certifiable with respect to $\Omega^*$, then for any $t > 0$,
  \begin{equation*}
    \Prob{|X - \Med(X)| > t} \leq 4\exp\left({-\frac{t^2}{4\change^2r(\Med(X) + t)}}\right) + 4\Prob{\Omega^*}
  \end{equation*}
\end{thm}

We then prove that the expectation and median are close as in the following lemma.
\begin{lemma}\label{expectation close to median}
If $X$ is $(r, \change)$-certifiable with respect to $\Omega^*$ and $M = \sup X$, then
\begin{equation*}
|\Expect{X} - \Med(X)| \leq 48\change\sqrt{r\Expect{X}} +  64r\change^2 + 4M\Prob{\Omega^*}.
\end{equation*}
\end{lemma}
\begin{proof}
Let $Y = X + \Expect{X}$.  Note that $\Expect{Y} - \Med(Y) = \Expect{X} - \Med(X)$, $\Med(Y) \geq \Expect{X} > 0$, and $\Expect{Y} \leq 2\Expect{X}$.
Note also that
\begin{equation*}
|\Expect{Y} - \Med(Y)| \leq \Expect{|Y - \Med(Y)|}.
\end{equation*}

Let $L = \lfloor M/(\change\sqrt{r\Med(Y)})\rfloor$, and note that $|Y - \Med(Y)| \leq  (L + 1)\change\sqrt{r\Med(Y)}$.  By partitioning the possible values of $|Y - \Med(Y)|$ into intervals of length $\change\sqrt{r\Med(Y)}$, we get

\begin{align*}
\Expect{|Y - \Med(Y)|} &\leq \begin{aligned}[t]
	\sum_{\ell=0}^L \change\sqrt{r\Med(Y)}(\ell + 1) &\left(\Prob{|Y - \Med(Y)| \geq \ell \change\sqrt{r\Med(Y)}}\right.\\
	&\left. - \Prob{|Y - \Med(Y)| \geq (\ell + 1) \change\sqrt{r\Med(Y)}}\right).\end{aligned}\\
&= \sum_{\ell=0}^L \change\sqrt{r\Med(Y)}\left(\Prob{|Y - \Med(Y)| \geq \ell \change\sqrt{r\Med(Y)}}\right).
\end{align*}
By applying Theorem \ref{exceptional talagrand's with median} with $t=\ell \change\sqrt{r\Med(Y)}$ to every summand,
\begin{equation*}
  \Expect{|Y - \Med(Y)|} \leq 4\change\sqrt{r\Med(Y)}\sum_{\ell=0}^L \left(\exp\left({-\frac{\ell^2\change^2r\Med(Y)}{4\change^2r(\Med(Y) + \ell \change\sqrt{r\Med(Y)})}}\right) + \Prob{\Omega^*}\right).
\end{equation*}
Note that for each $\ell\in\{0, \dots, L\}$,
\begin{multline*}
  \exp\left({\frac{\ell^2\change^2r\Med(Y)}{4\change^2r(\Med(Y) + \ell \change\sqrt{r\Med(Y)})}}\right) \leq \exp\left({\frac{\ell^2\change^2r\Med(Y)}{8\change^2r\max\{\Med(Y), \ell \change\sqrt{r\Med(Y)}\}}}\right)\\
  \leq \exp\left({\frac{\ell^2\change^2r\Med(Y)}{8\change^2r\Med(Y)}}\right) + \exp\left({\frac{\ell^2\change^2r\Med(Y)}{8\change^3r\ell\sqrt{r\Med(Y)}\}}}\right) = \exp\left(\ell^2/8\right) + \exp\left(\frac{\ell\sqrt{\Med(Y)}}{8\change\sqrt{r}}\right).
\end{multline*}
Note also that
\begin{equation*}
  4\change\sqrt{r\Med(Y)}\sum_{\ell=0}^L\Prob{\Omega^*} \leq 4M\Prob{\Omega^*}.
\end{equation*}
Therefore
\begin{equation*}
  \Expect{|Y - \Med(Y)|} \leq 4\change\sqrt{r\Med(Y)}\sum_{\ell=0}^\infty \left(\exp\left({-\ell^2/8}\right) + \exp\left({-\frac{\ell\sqrt{\Med(Y)}}{8\change\sqrt{r}}}\right)\right) + 4M\Prob{\Omega^*}.
\end{equation*}
Note that $\sum_{\ell=0}^\infty e^{-\ell x} = \frac{1}{1 - e^{-x}}$.  Note also that $\frac{x}{2} \leq 1 - e^{-x}$ if $x < \frac{3}{2}$.   Since $\frac{1}{1 - e^{-x}} < 2$ when $x\geq \frac{3}{2}$, $\frac{1}{1 - e^{-x}} \leq \max\{2, \frac{2}{x}\}$.  Therefore
\begin{equation*}
  \sum_{\ell=0}^\infty \exp\left(-\frac{\ell\sqrt{\Med(Y)}}{8\change\sqrt{r}}\right) \leq \max\left\{2, \frac{16\change\sqrt{r}}{\sqrt{\Med(Y)}}\right\}.
\end{equation*}

Note that $\sum_{\ell=0}^\infty e^{-\ell^2/ 8} < 4$.  Therefore
\begin{equation*}
  \Expect{|Y - \Med(Y)|} \leq 4\change\sqrt{r\Med(Y)}\left(4 + \max\left\{2, \frac{16\change\sqrt{r}}{\sqrt{\Med(Y)}}\right\}\right) + 4M\Prob{\Omega^*}.  
\end{equation*}
Since the maximum of two numbers is at most their sum,
\begin{equation*}
  \Expect{|Y - \Med(Y)|} \leq 24\change\sqrt{r\Med(Y)} + 64r\change^2 + 4M\Prob{\Omega^*}.
\end{equation*}
Since $\Med(Y) \leq 2\Expect{Y} \leq 4\Expect{X}$,
\begin{equation*}
\Expect{Y - \Med(Y)|} \leq 48\change\sqrt{r\Expect{X}} +  64r\change^2 + 4M\Prob{\Omega^*},
\end{equation*}
as desired.
\end{proof}

Lemma~\ref{expectation close to median} is similar to Fact 20.1 in~\cite{MR00}.  However, the proof of Fact 20.1 is flawed, as we now describe.  Molloy and Reed upper bound $\Prob{|X - \Med(X)| > ic\sqrt{r\Med(X)}}$ by $4e^{-i^2/8}$ for every positive integer $i$ using Talagrand's Inequality I; however, Talagrand's Inequality I only applies if $0 \leq ic\sqrt{r\Med(X)} \leq \Med(X)$.  Our proof of Lemma~\ref{expectation close to median} avoids this flaw, since Theorem~\ref{exceptional talagrand's with median} has no restriction on $t$.  However, for these large values of $i$, we bound this probability by $\exp\left(-\frac{\ell\sqrt{\Med(Y)}}{8\change\sqrt{r}}\right)$ instead, which leads to the additional $64r\change^2$ term.

Now we can prove Theorem \ref{exceptional talagrand's} assuming Theorem \ref{exceptional talagrand's with median}.
\begin{proof}[Proof of Theorem \ref{exceptional talagrand's}]
Since $t > 96\change\sqrt{r\Expect{X}} +  128r\change^2 + 8M\Prob{\Omega^*},$
\begin{equation}\label{t over 2 bound}
\frac{t}{2} > 48\change\sqrt{r\Expect{X}} +  64r\change^2 + 4M\Prob{\Omega^*}.
\end{equation}
By applying Lemma \ref{expectation close to median} and then \eqref{t over 2 bound},
\begin{equation*}
\Prob{|X - \Expect{X}| > t} \leq \Prob{|X - \Med(X)| > \frac{t}{2}}.
\end{equation*}
Since $\Med(X) \leq 2\Expect{X}$, Theorem \ref{exceptional talagrand's with median} implies that
\begin{align*}
\Prob{|X - \Med(X)| > \frac{t}{2}} &\leq 4\exp\left({-\frac{(t/2)^2}{4\change^2r(2\Expect{X} + (t/2))}}\right) + 4\Prob{\Omega^*},\\
 &=4\exp\left({-\frac{t^2}{8\change^2r(4\Expect{X} + t)}}\right) + 4\Prob{\Omega^*}
\end{align*}
as desired.
\end{proof}

It remains to prove Theorem \ref{exceptional talagrand's with median}.

Let $((\Omega_i, \Sigma_i, \mathbb P_i))_{i=1}^n$ be probability spaces and $(\Omega, \Sigma, \mathbb P)$ their product space.  For a set $A\subseteq \Omega$ and event $\omega\in\Omega$, let
\begin{equation}
d(\omega, A) = \sup_{||\alpha||=1}\left\{\tau : \sum_{i:\omega_i\neq\omega'_i}\alpha_i\geq\tau \rm{\ for\ all\ }\omega'\in A\right\}.
\end{equation}

We use the original version of Talagrand's Inequality.
\begin{thm}[Talagrand's Inequality \cite{T95}]\label{OG Tala}
If $A,B\subseteq \Omega$ are measurable sets such that for all $\omega\in B$, $d(\omega, A) \geq \tau$, then
$$\Prob{A}\Prob{B}\leq e^{\frac{-\tau^2}{4}}.$$
\end{thm}

We can now prove Theorem~\ref{exceptional talagrand's with median}.
\begin{proof}[Proof of Theorem \ref{exceptional talagrand's with median}]
It suffices to show that
\begin{equation}\label{one sided estimation}
\Prob{X\leq \Med(X) - t} \leq 2\exp\left({-\frac{t^2}{8r\change^2(\Med(X) + t)}}\right) + 2\Prob{\Omega^*}
\end{equation}
and
\begin{equation}\label{other one sided estimation}
\Prob{X\geq \Med(X) + t} \leq 2\exp\left({-\frac{t^2}{8r\change^2(\Med(X) + t)}}\right) + 2\Prob{\Omega^*}.
\end{equation}
Let 
\begin{align*}
&A = \{\omega\in \Omega\backslash \Omega^* : X(\omega) \geq \Med(X) + t\}, \text{ and}\\
&B = \{\omega\in \Omega\backslash \Omega^* : X(\omega) \leq \Med(X) \}.
\end{align*}

We need to show the following.
\begin{claim}\label{A and B satisfy OG Tala}
For all $\omega\in B$, $d(\omega, A) \geq \frac{t}{c\sqrt{r(\Med(X) + t)}}$.
\end{claim}
To that end, let $\omega'\in A$.  Since $X$ is $(r, \change)$-certifiable, there exists an $(r, \change)$-certificate, $I$, for $X, \omega', \Med(X) + t,$ and $\Omega^*$.  Thus, the outcomes $\omega$ and $\omega'$ differ in at least $t/\change$ coordinates of $I$.  Therefore if we set $\alpha = 1 / \sqrt{|I|}\cdot\mathbf 1_I$ where $\mathbf 1_I$ is the characteristic vector of $I$, then $\omega$ and $\omega'$ have $\alpha$-hamming distance at least $t / (\change\sqrt{r(\Med(X) + t)})$.  Hence, the claim follows.

Now \eqref{other one sided estimation} follows from Claim \ref{A and B satisfy OG Tala} and Theorem \ref{OG Tala}.  The proof of \eqref{one sided estimation} is similar, so we omit it.
\end{proof}

The proof of Claim~\ref{A and B satisfy OG Tala} demonstrates why we introduce $k$ into the definition of $(r, \change)$-certificates, rather than considering changing the outcome of only one trial.  We may change the outcome of one trial and obtain an exceptional outcome, in which case we need that changing the outcome of yet another trial does not greatly affect $X$, or else the outcomes $\omega$ and $\omega'$ may differ for only two trials.

\end{document}